\documentclass[11pt]{article}
\usepackage{amsfonts,amssymb,amsthm,eucal,amsmath,verbatim}
\usepackage{setspace}

\usepackage[english]{babel}
\usepackage{tikz}
\usepackage[letterpaper,top=1in,bottom=1in,left=1in,right=1in,marginparwidth=1.75cm]{geometry}
\usepackage{subfigure}

\usepackage{mdframed}
\usepackage{graphicx}
\usepackage{mathtools}
\usepackage[colorlinks=true, allcolors=blue]{hyperref}
\usepackage{url} 
\usepackage{subfiles} 
\usepackage{amsfonts}
\usepackage{amsthm}
\usepackage{subcaption}
\usepackage[section]{placeins}
\usepackage{bbm} 
\usepackage{enumitem}
\usepackage{lmodern}
\usepackage{textcomp} 

\usepackage{mathtools}

\newtheorem{theorem}{Theorem}[section]
\newtheorem{corollary}[theorem]{Corollary}
\newtheorem{lemma}[theorem]{Lemma}

\newtheorem{proposition}[theorem]{Proposition}

\usepackage[normalem]{ulem}

\numberwithin{equation}{section}
\numberwithin{theorem}{section}

\title{Mesoscopic Rates of Convergence for Hermitian Unitary Ensembles}
\date{}
\author{Mengchun Cai%
\thanks{Department of Mathematics, Applied Mathematics, and
  Statistics, Case Western Reserve University, Cleveland, Ohio,
  U.S.A.; mengchun.cai@case.edu.}%
\qquad
Kyle Taljan%
\thanks{Department of Mathematics, Applied Mathematics, and
  Statistics, Case Western Reserve University, Cleveland, Ohio,
  U.S.A.; kyle.taljan@case.edu.}}

\begin{document}
\maketitle
\begin{abstract}
\noindent This paper provides mesoscopic rates of convergence (ROC) with respect to the $L^1$-Wasserstein distance for the eigenvalue determinantal point processes (DPPs) from the three major Hermitian unitary ensembles, the Gaussian Unitary Ensemble (GUE), the Laguerre Unitary Ensemble (LUE), and the Jacobi Unitary Ensemble (JUE) to their limiting point processes. We prove ROCs for the bulk of the GUE spectrum, the hard edge of the LUE spectrum, and the soft edges of the GUE, LUE, and JUE spectrums. These results are called mesoscopic because we are able to directly compare the point counts between the converging and limit DPPs in a range of scales. We are able to achieve these results by controlling the trace class norm of the integral operators determined by the DPP kernels.
\end{abstract}
\
\section{Introduction}

\subsection{Rates of Convergence and Mesoscopic Information}

The literature on random matrix theory (RMT) is huge and varied. Naturally, many of the most important questions revolve around the (random) spectrum of different classes of random matrices. As an example, consider a $N\times N$ Hermitian matrix $X:=\frac{1}{2}(Y+Y^*)$ where $Y$ is a $N\times N$ matrix with all entries of the i.i.d. complex Gaussian $\mathcal{N}_\mathbb{C}(0,1)$. It is well known that $\mu_N^X:=\frac{1}{N}\sum\limits_{k=1}^N\delta_{\lambda_k(X)}$, the empirical distribution of the eigenvalues of $X$, as $N$ goes to $\infty$, almost surely converges towards the semicircle law $\mu_{sc}(dx)=\frac{1}{2\pi}\sqrt{4-x^2}\chi_{[-2,2]}(x)dx$ in distribution (see Bai \cite{bai1993convergence} for more details). 

While asymptotic statements are standard, there is a major, more modern research thrust toward quantitative rate of convergence (ROC) results across probability and analysis. An archetype of this sort of work is the Lindeberg-Levy Central Limit Theorem (CLT) versus the Berry-Esseen CLT. The first says that a properly centered and scaled sum of random variables converges in distribution to a limit, whereas the second provides a quantitative rate at which the limiting distribution is achieved (interested readers are encouraged to refer to \cite{fischer_history_2011} as a comprehensive introduction). As a similar example in RMT, in multivariate statistics and quantum mechanics, as in \cite{agranov_airy_2020} and \cite{mehta_random_2004}, the Airy distribution is a limiting distribution in a number of important situations. However, concrete applications in many fields depend on the distribution for large, finite dimension $N$. Non-asymptotic results allow estimation of the error, with respect to $N$, when utilizing the Airy distribution for practical inferences. So, in RMT, as soon as a theorem like convergence to the semicircle law is proved, one can ask for the rate, as a function of $N$, of the convergence. 

Besides seeking non-asymptotic (or, perhaps better, beyond asymptotic) results, this work will also be concerned with proving results on the most refined scales possible. In our case we will have an ensemble of random $N \times N $ matrices as $N$ grows, which gives a random point process of the $N$ eigenvalues (for this paper, all eigenvalue point processes will be on $\mathbb{R}$). Many results in RMT, for example, the convergence to the semicircle law above, are macroscopic, in that they look at the entire spectrum and the continuous limit that the (properly recentered and rescaled) eigenvalues approach (we cite \cite[Chapter $2\&3$]{bai1993convergence} as a good reference). Some other results are microscopic as they focus on the deviation of one single eigenvalue around its empirical location (see \cite[Chapter $5$]{bai1993convergence} for the extreme eigenvalue and \cite{erdos_rigidity_2012}, \cite{pillai_universality_2014} for the general ones). However, here we will be concerned with mesocropic results, that is, results that focus on individual (random) eigenvalue counting functions arising from DPPs on a set with a variable size. Unlike the averaging process in \cite{dallaporta_eigenvalue_2012} and \cite{dallaporta2013eigenvaluevarianceboundscovariance}, here we are able to prove results about individual point counts from our point processes, the most detailed information possible in the setting.

The reason we are able to produce such detailed results is because the (random) spectrum of certain classes of random matrices come from a particularly amenable class of point processes known as Determinantal Point Processes (DPPs). As introduced in  \cite{log-gas}, the DPPs in RMT are a type of $\beta$-ensemble, specifically where $\beta=2$. In RMT, $\beta$-ensembles by themselves are already good settings for refined results (see \cite{dumitriu2002matrix} and  \cite{dumitriu_eigenvalues_2005}). DPPs, then, hold a position similar to $L^2$ spaces versus $L^p$ spaces more generally, and, as with the Hilbert space $L^2$, DPPs are usually the appropriate setting for proving the most precise results possible in all of RMT. We will briefly review some essential features of DPPs in the next section to show why this is the case.


\subsection{Determinantal Point Processes}


Let us briefly review the definition of DPPs.

A point process $\mathfrak{X}$ in a locally compact Polish space $\Omega$ is a random discrete subset of $\Omega$. For any $A\subset\Omega$, the integer-valued function
\begin{equation*}
\mathcal{N}_\mathfrak{X}(A):=\#\{\omega\in\Omega:\omega\in A\}
\end{equation*}
is called the counting function of $\mathfrak{X}$. In this paper, we mainly focus on point processes in $\Omega=\mathbb{R}$. Similarly, given a Borel $A\subset\mathbb{R}$ and a random Hermitian matrix $H$, $\mathcal{N}_H(A)$ is defined to be equal to the number of eigenvalues of $H$ in $A$. 

Fix a point process $\mathfrak{X}$. If they exist, the \textbf{correlation functions} (also called joint intensities) for $\mathfrak{X}$ are a sequence of locally integrable functions $\{\rho_k:\mathbb{R}^k\rightarrow\mathbb{R} \}_{k=1}^{\infty}$ satisfying the following condition: for all mutually disjoint measurable subsets $\{D_j\}_{j=1}^n$ of $\mathbb{R}$, 
\begin{equation*}
	\mathbb{E} \left[\prod_{j=1}^{n} \mathcal{N}_{\mathfrak{X}}(D_j)\right] = \int_{\prod_j D_j} \rho_n(x_1,\dots,x_n) dx_1 \cdots dx_n,
\end{equation*}   
where the integral is with respect to Lebesgue measure. Analogous to the classic moment problem from probability, under many circumstances correlation functions uniquely specify a point process (see \cite{soshnikov2000determinantal} where it is shown that this is the case for DPPs).

Furthermore, if there exists a function $K(x,y)$ such that, for almost every $x_i\in\mathbb{R}$,
\begin{equation*}
    \rho_n(x_1,\dots,x_n)=\det\left(K(x_i,x_j)_{i,j=1}^n\right)
\end{equation*}
for any $n$, then $\mathfrak{X}$ is called a \textbf{determinantal~point~process}, and $K$ is called the \textbf{kernel} of $\mathfrak{X}$. 

The kernel function ends up specifying two things. It specifies the correlation functions as just described, but it also specifies an integral operator. As will be the case here, a problem involving a DPP can usually be translated to a problem about the operator determined by the kernel function. Dealing with the integral operators tends to be easier, often because the kernels that determine the operators are typically well-studied, classical special functions, and so have many results that can be used off the shelf.

\subsection{The GUE, LUE, and JUE}
Now let us be specific about the unitary ensembles addressed in this paper. The joint distribution of the eigenvalues of the Hermitian matrix $G_N:=\frac{1}{2}(Y+Y^*)$ where $Y$ is a $N\times N$ matrix with all entries complex standard Gaussian is known as a \textbf{Gaussian~Unitary~Ensemble~(GUE)}. Let $X$ be a $(N+a)\times N$ matrix with iid complex standard Gaussian entries. We say the joint distribution of the eigenvalues of $L_{a,N}:=X^*X$ is a \textbf{Laguerre~Unitary~Ensemble~(LUE)} with parameter $a>-N$. Consider the matrix $J_{a,b,N}:=(A^*A)(A^*A+B^*B)^{-1}$ where $A$ is a $(N+a)\times N$ matrix, $B$ is a $(N+b)\times N$ matrix and all entries in $A$ and $B$ are iid complex standard Gaussian. The joint distribution of the eigenvalues of $J_{a,b,N}$ is a \textbf{Jacobi~Unitary~Ensemble~(JUE)} with parameters $a\ge 0$ and $b\ge 0$. Those parameters $a$, $b$ in the LUE and JUE can be chosen as some sequences $a(N)$ and $b(N)$ determined by $N$.

It is shown by Forrester, in Chapter $5$ of \cite{log-gas} that the GUE, LUE and JUE are all determinantal point processes with kernel:
\begin{equation}\label{kernel1}
    K_{N}(x,y)=\frac{\left(w(x)w(y)\right)^{\frac{1}{2}}}{(p_{N-1},p_{N-1})_2}\frac{p_{N}(x)p_{N-1}(y)-p_{N}(y)p_{N-1}(x)}{x-y}
\end{equation}
where $p_j$'s are the $j$th Hermite, Laguerre and Jacobi polynomials defined on different domains respectively, $w(x)$'s are corresponding weight factors and
\begin{equation*}
    (p_j,p_j)_2:=\int_{\mathbb{R}}p_j^2(x)w(x)dx.
\end{equation*}
The detailed information is listed as following.
\begin{center}
\begin{tabular}{c c c c} 
 \hline
  & $w(x)$ & $p_n(x)$ & $(p_n,p_n)_2$ \\[2ex] \hline
 $G_n$ & $e^{-x^2}$ & $2^{-n}H_n(x)$ & $\sqrt\pi2^{-n}n!$ \\[2ex] 
 \hline
 $L_{a,n}$ & $x^ae^{-x}\chi_{}(0,\infty)$ & $(-1)^nn!L_n^a(x)$ &$\Gamma(n+1)\Gamma(a+n+1)$   \\[2ex]
 \hline
 $J_{a,b,n}$ &$(1-x)^a(1+x)^b\chi_{(-1,1)}$ & $2^nn!\frac{\Gamma(a+b+n+a)}{\Gamma(a+b+2n+1)}P^{(a,b)}_n(x)$ &$2^{a+b+2n+1}\Theta(a,b,n)$  \\[2ex]
 \hline
\end{tabular}
\end{center}
where
\begin{equation}\label{kernel2}
    \begin{array}{ll}
         H_n(x)=\sum\limits_{k=0}^{\lfloor n/2\rfloor}(-1)^k2^{n-k}\left(\begin{matrix}
             n\\ 2k
         \end{matrix}\right)\frac{(2k)!}{2^kk!}x^{n-2k},\\ \\
        L_n^a(x)=\sum\limits_{k=0}^n(-1)^k\left(\begin{matrix}
             n+a\\n-k
         \end{matrix}\right)\frac{x^k}{k!},\\\\
         P_n^{(a,b)}(x)=\frac{\Gamma(a+n+1)}{n!\Gamma(a+b+n+1)}\sum\limits_{k=0}^n\left(\begin{matrix}
             n\\k
         \end{matrix}\right)\frac{\Gamma(a+b+n+k+1)}{\Gamma(a+k+1)}\left(\frac{x-1}{2}\right)^k
    \end{array}
\end{equation}
are corresponding polynomials and
\begin{equation*}
    \Theta(a,b,n):=\frac{\Gamma(n+1)\Gamma(a+n+1)\Gamma(b+n+1)\Gamma(a+b+n+1)}{\Gamma(a+b+2n+1)\Gamma(a+b+2n+2)}.
\end{equation*}
A determinantal point process $\mathfrak{X}_{Sine}$ given by the kernel:
\begin{equation}\label{sinek}
    K_{Sine}(x,y)=\frac{\sin(\pi(x-y))}{\pi(x-y)}
\end{equation}
is called a \textbf{sine~point~process}. This process is the limiting process of our ensembles around the bulk after proper centering and rescaling. In another word, those eigenvalues away from edges behave similarly to the sine point process when $N$ is large. 

At the right edge, the LUE and JUE have similar limiting behavior after proper rescaling. They all converge to the \textbf{Airy~point~process} $\mathfrak{X}_{Ai}$ given by the kernel:
\begin{equation*}
    K_{Ai}(x,y)=\frac{Ai(x)Ai'(y)-Ai(y)Ai'(x)}{x-y}
\end{equation*}
where
\begin{equation*}
    Ai(x)=\frac{1}{\pi}\int_0^\infty\cos\left(\frac{t^2}{3}+xt\right)dt
\end{equation*}
is the \textbf{Airy~function}. The right edge is also called the \textbf{soft edge} as with probability $1$, there is no eigenvalue exceeding it (see Geman \cite{geman_limit_1980}). Around $0$, the LUE has some unique behavior once  $\lim\limits_{N\rightarrow\infty }a(N)=a$, or equivalently, $a(N)$ is a fixed integer for $N$ large as $a(N)\in\mathbb{N}$ for any $N$. In this case, the left edge for the limiting spectral distribution of $L_{a,N}$ touches $0$, which is a natural bound for eigenvalues of $L_{a,N}$ as it is a positive semidefinite matrix. Under this circumstance, $0$ is called the \textbf{hard~edge}. After proper rescaling, the hard edge of the LUE converges to the \textbf{Bessel~point~process} $\mathfrak{X}_{Bes,a}$, whose kernel is
\begin{equation*}
    K_{Bes,a}(x,y)=\frac{\sqrt x J_a(\sqrt x)J'_a(\sqrt y)-\sqrt y J_a(\sqrt y)J_a'(\sqrt x)}{2(x-y)}
\end{equation*}
where
\begin{equation*}
    J_a(x)=\sum_{m=0}^\infty\frac{(-1)^m}{m!\Gamma(m+a+1)}\left(\frac{x}{2}\right)^{2m+a}
\end{equation*}
is the \textbf{Bessel~function} with parameter $a$.

By convention, we refer to points in a DPP as eigenvalues even when they are not true eigenvalues as they are in the sine, Airy and Bessel point processes.

\subsection{Qualitative Statement of Results}
Given a Polish metric space $(\Omega,d)$ and two probability measures $\mu$ and $\nu$ on $\Omega$, the $L^p$-Wasserstein distance $W_p$ between $\mu$ and $\nu$ is defined as:
\begin{equation*}
    W_p(\mu,\nu):=\\\inf_{(x,y)\in\pi(\mu,\nu)}\left(\mathbb{E}d^p(x,y)\right)^{\frac{1}{p}}
\end{equation*}
where $\pi(\mu,\nu)$ contains all couplings of $\mu$ and $\nu$ and $p\ge1$. This metric is widely used in probability theory and in particular, in optimal transport (see \cite[Chapter $6$]{villani2009optimal}). A key property is that for any $1\le p\le q$, the inequality $W_p\le W_q$ holds. Therefore, among the family of $L^p$- Wasserstein distances, $W_1$ induces the weakest topology. In this paper, the choice of $W_1$ is motivated a powerful lemma that serves as the starting point for our proofs. 

For the limiting behaviors mentioned above, our goal in this paper is to provide a ROC for different ensembles with respect to the $L^1$-Wasserstein distance $W_1$.

A brief word on terminology: in saying things like ``The LUE converges to Bessel kernel point process" we mean ``under a suitable recentering and scaling (to be specified) the DPP formed by the LUE eigenvalues converges to the DPP determined by the Bessel kernel". 

Here is a qualitative statement of our results. The GUE (eigenvalue DPP) under a bulk scaling converges to the sine kernel DPP  with rate $N^{-1}$. For the GUE, LUE  and JUE, we care about their convergence at two edges. At the soft edge, after proper scaling, all of the GUE, LUE and JUE converge to the Airy kernel point process with rate $N^{-\frac{2}{3}}$. At the hard edge, the scaled LUE converges to the Bessel point process with rate $N^{-2}$. We conjecture all of these rates to be optimal and a short discussion will be given in section $2.2$.

In each case the type of convergence will be a direct comparison of point counts between the two DPPs, which, again, will be the most precise type of information available. Qualitatively, all our results can be written in the form

\begin{equation*}
    W_{1}(\mathcal{N}_{N},\mathcal{N}_{\infty})
	\leq \frac{f(s)}{N^p},
\end{equation*}
where $\mathcal{N}_{N}$ and $\mathcal{N}_{\infty}$ are the (eigenvalue) counting functions on some specific set $A\subset\mathbb{R}$ while $\mathcal{N}_N$ is from the eigenvalue DPP and $\mathcal{N}_\infty$ from limiting DPP, $f(s)$ will be some function on $A$, and the GUE bulk will have $p=1$, the LUE hard edge will have $p=2$, and all three soft edge cases will have $p=\frac{2}{3}$. 

The four edge cases (LUE hard edge, and all three of the soft edge cases) will immediately yield the best known ROCs for the extreme eigenvalues to their known limiting distributions, something that points to the utility of the detailed results proved here.

\subsection{Outline of Proof}

Recall that for a kernel function $K(x,y)$ on some domain $D^2\subset\mathbb{R}^2$, we can define a corresponding integral operator $\mathcal{K}$ by the formula
\begin{equation*}
    \mathcal{K}(f)(x)=\int_D K(x,y)f(y)dy
\end{equation*}
for some function $f$ defined on $D$. Recall further that there is a hierarchy of classes of such (compact) operators (again analogous to $L^p$ spaces) based on the rate of decay of the singular values of the operator known as the Schatten $p$-Norms. Given a compact operator $T$, it belongs to the Schatten $p$-class if
\begin{equation*}
 \lVert T \rVert_{Sp}:= \left( \sum_{n=1}^{\infty} s_n(T)^p\right)^{\frac{1}{p}} < \infty
\end{equation*}
where $s_n(T)$'s are singular values of $T$. For convenience, for a given suitable operator $L$, we simply denote its Schatten $p$-norm as $\|L\|_p$ in this paper.

As expected, the two most important cases are $p=1$ and $p=2$. The $p=1$ case are known as trace class operators and the $p=2$ case are called Hilbert-Schmidt operators.

Soshnikov proved in \cite{soshnikov2000determinantal} that the kernel of self-adjoint, locally trace class operator defines a determinantal point process if and only if all the eigenvalues of the operator are between $0$ and $1$.

The motivation and method for this paper come from the joint work of E. Meckes and M. Meckes in \cite{meckes2015self}. Their proof of the proposition $4$ can be generalized to the following useful lemma:


\begin{lemma}\label{M and M}
	Consider two DPPs $\mathfrak{X}$ and $\widetilde{\mathfrak{X}}$ with Hermitian kernels $K(x,y)$ and $\widetilde{K}(x,y)$ and associated integral operators $\mathcal{K}$ and $\widetilde{\mathcal{K}}$. Assume the integral operators are trace class. Then,
	\begin{equation*}
	 W_1(\mathcal{N}_\mathcal{X},\mathcal{N}_{\tilde{\mathcal{X}}})
	\leq \lVert \mathcal{K} - \widetilde{\mathcal{K}} \rVert_1.
	\end{equation*}
\end{lemma}
\begin{proof}
    We only need to consider the second inequality. Let $\{\lambda_j(\mathcal{K})\}$ and $\{\lambda_j(\tilde{\mathcal{K}})\}$ be the eigenvalues of $\mathcal{K}$ and $\tilde{\mathcal{K}}$ respectively. As both $\mathcal{X}$ and $\tilde{\mathcal{X}}$ are DPPs, we know all eigenvalues are between $0$ and $1$ and hence, $\lambda_j(\mathcal{K})=s_j(\mathcal{K})$ and $\lambda_j(\tilde{\mathcal{K}})=s_j(\tilde{\mathcal{K}})$ for any $j$. Without the loss of generality, we assume all eigenvalues are listed nonincreasingly.
    
    Let $\{B_j\}_{j\in\mathbb{N}}$ and $\{\tilde{B}_j\}_{j\in\mathbb{N}}$ be two sequences of independent Bernoulli random variables such that 
    \begin{equation*}
        \mathbb{P}\{B_j=1\}=\lambda_j(\mathcal{K}),~~\mathbb{P}\{\tilde{B}_j=1\}=\tilde{\lambda}_j(\mathcal{K}).
    \end{equation*}
     According to \cite[Theorem $7$]{hough2006determinantal}, we have $\mathcal{N}_\mathcal{X}\overset{d}{=}\sum_{j=1}^\infty B_j$ and $\mathcal{N}_{\tilde{\mathcal{X}}}\overset{d}{=}\sum_{j=1}^\infty\tilde{B}_j$ (we can set all eigenvalues $0$ behind some $N\in\mathbb{N}$ if the DPP is finite). Let $\{Y_n\}_{n\in\mathbb{N}}$ be a sequence of iid random variables uniformly distributed on $[0,1]$. For each $j\in\mathbb{N}$, we further define $\xi_j:=\chi_{\{Y_j\le\lambda_j(\mathcal{K})\}}$ and $\tilde{\xi}_j:=\chi_{\{Y_j\le\lambda_j(\tilde{\mathcal{K}})\}}$, then $\left(\sum_j\xi_j,\sum_j\tilde{\xi}_j\right)$ is a coupling of $\mathcal{N}_\mathcal{X}$ and $\mathcal{N}_{\tilde{\mathcal{X}}}$. By the definition of $W_1$, we see
    \begin{equation*}
    \begin{split}
W_1(\mathcal{N}_{\mathcal{X}},\mathcal{N}_{\tilde{\mathcal{X}}})&\le\mathbb{E}\left|\sum_{j=1}^\infty\xi_j-\tilde{\xi}_j\right|\le\sum_{j=1}^\infty\mathbb{E}|\xi_j-\tilde{\xi}_j|=\sum_{j=1}^\infty\mathbb{E}\chi_{\{Y_j\in|\lambda_j(\mathbf{K})-\lambda_j(\tilde{\mathbf{K}})|\}}\\
        &=\sum_{j=1}^\infty|\lambda_j(\mathcal{K})-\lambda_j(\tilde{\mathcal{K}})|=\sum_{j=1}^\infty|s_j(\mathcal{K})-s_j(\tilde{\mathcal{K}})|
    \end{split}
    \end{equation*}
Applying the generalized Mirsky inequality (\cite[Theorem $5.4$]{markus1964eigen}) to the rightmost side above, we further get
\begin{align*}
    &\sum_{j=1}^\infty|s_j(\mathcal{K})-s_j(\tilde{\mathcal{K}})|\le\sum_{j=1}^\infty|s_j(\mathcal{K}-\tilde{\mathcal{K}})|=\|\mathcal{K}-\tilde{\mathcal{K}}\|_1.&\qedhere
\end{align*}
\end{proof}

In every case our method follows the same pattern. We start with the above lemma translating our problem to one of estimating a trace class norm. In the three soft edge cases, the trace class estimates we want are already available in  \cite{el2006rate}, \cite{johnstone2008multivariate}, and \cite{johnstone2012fast}. For the GUE bulk and the LUE hard edge, we prove the required estimates below.

\section{Statement of Results}\label{Sec: Results}
\subsection{Exact Statements}
Restricting all counting functions to some set $A\subset\mathbb{R}$,
the following results will be shown in this paper.
\begin{theorem}\label{G bulk}
    (GUE bulk) Let $I=[-s,s]$ for $s>0$ be an interval and $\mathcal{X}_{G_N}^{Bulk}$ be the bulk-scaled eigenvalue process of a GUE $G_N$ with kernel $K_{G_N}^{Bulk}=\frac{1}{\sqrt{2N}}K_{G_N}\left(\frac{x}{\sqrt{2N}},\frac{y}{\sqrt{2N}}\right)$. On any $A\subset I$, assume $N\geq 2$,
	then, there exists a uniform constant $C$ such that
	\begin{equation}\label{main bulk}
    W_{1}(\mathcal{N}_{G_N}^{Bulk},\mathcal{N}_{Sine})
	\leq \frac{C\max{\{s^{10},1\}}}{N}.
	\end{equation}
\end{theorem}
\begin{theorem}\label{L hard}
    (LUE hard edge) For a fixed non-negative $a\in \mathbb{Z}$, define $\tau_N:=1-\frac{a}{2N}$. Let $I=[0,s]$ for $s>0$ be an interval and $\mathcal{X}_{L_{a,N}}^{Hard}$ be the hard edge-scaled eigenvalue process of an LUE $L_{a,N}$ with the kernel $K_{L_{a,N}}^{Hard}(x,y):=\frac{\tau_N}{4N} K_{L_{a,N}}\left(\frac{\tau_Nx}{4N},\frac{\tau_Ny}{4N}\right)$. On $A\subset I$, when $N$ is large enough, we have
    \begin{equation}\label{main L hard}
        W_1(\mathcal{N}_{L_{a,N}}^{Hard},\mathcal{N}_{Bes,a})\le\frac{h(s)}{N^2}
    \end{equation}
where $h(s)$ is a non-negative function bounded by $s^{a+1}$ when $s$ is small.
\end{theorem}
 Theorem \ref{L hard} also provides us with a bound for the distribution of the least eigenvalue $\lambda_{\min}$ of an LUE as following.    
\begin{corollary}\label{smallest}
    Let $\lambda_{min}^N$ denote the hard-edge-scaled least eigenvalue of the LUE $L_{a,N}$ and $B_a(s)$ be the distribution function of the least eigenvalue of the Bessel point process with parameter $a$. For large $N$, we have
    \begin{equation*}
        |\mathbb{P}\{\lambda_{min}^N\le s\}-B_a(s)|\le \frac{h(s)}{N^{2}}.
    \end{equation*}
    where $h$ is the same function as in Theorem \ref{L hard}.
\end{corollary}
\begin{proof} By the facts that
\begin{equation*}
    \mathbb{P}\{\lambda_{min}^N\le s\}=1-\mathbb{P}\{\lambda_{min}^N>s\}=1-\mathbb{P}\{\mathcal{N}_{L_{a,N}}^{Hard}([0,s])=0\}
\end{equation*}
and
\begin{equation*}
    B_a(s)=1-\mathbb{P}\{\mathcal{N}_{Bes,a}([0,s])=0\},
\end{equation*}
from (\ref{main L hard}), we can see
\begin{equation*}
\begin{split}
     |\mathbb{P}\{\lambda_{min}^N\le s\}-B_a(s)|&=|\mathbb{P}\{\mathcal{N}_{Bes,a}([0,s])=0\}-\mathbb{P}\{\mathcal{N}_{L_{a,N}}^{Hard}([0,s])=0\}|
     \\&\le d_{TV}(\mathcal{N}_{L_{a,N}}^{Hard}([0,s]),\mathcal{N}_{Bes,a}([0,s]))
     \\&\le W_1(\mathcal{N}_{L_{a,N}}^{Hard}([0,s]),\mathcal{N}_{Bes,a}([0,s]))\\&\le\frac{h(s)}{N^2}
\end{split}
\end{equation*}
when $N$ is large.
\end{proof}

The distribution $B_a(s)$ above is also called a {\bf hard-edge~Tracy-Widom~distribution} with parameter $a$; see \cite{tracy1994level} for the explicit expression given by Tracy and Widom. In particular, when $a=0$, we have $B_0(s)=e^{-s}$ and hence, the limiting distribution of the scaled least eigenvalue of the LUE is the exponential distribution, which was first given by Edelman in \cite{edelman_eigenvalues_1988}.

\begin{theorem}\label{G soft}
    (GUE soft edge) Let $I\subset[s,+\infty)$ for $s>-\infty$. Define
    \begin{equation*}
        t_{G_N}(x):=\frac{1}{2}\left(\sqrt{2N+1}+\sqrt{2N-1}\right)+2^{-\frac{1}{2}}N^{-\frac{1}{6}}x:=\mu_{G_N}+\sigma_{G_N}x
    \end{equation*}
for any $x\in I$. Let $\mathcal{X}_{G_{N}}^{Soft}$ denote the soft edge-scaled point process for GUE $G_N$ with the kernel $K_{G_N}^{Soft}(x,y):=\sigma_{G_N}K_{G_N}(t_{G_N}(x),t_{G_N}(y))$. Then on $A\subset I$,
\begin{equation}\label{main G soft}
    W_1(\mathcal{N}_{G_N}^{Soft},\mathcal{N}_{Ai})\le\frac{C(s)}{N^{\frac{2}{3}}}
\end{equation}
for $C(s)$ a non-increasing function of $s$.
\end{theorem}
\begin{theorem}\label{L soft}
    (LUE soft edge) Consider an LUE $L_{a,N}$ such that $\lim\limits_{N\rightarrow\infty}\frac{N+a(N)}{N}=\gamma\in[1,\infty)$. Let $I\subset[s,+\infty)$ for $s>-\infty$. Define parameters:
    \begin{equation*}
      \begin{array}{ll}
      \tilde{\mu}_{n,m}:=\left(\sqrt{n+\frac{1}{2}}+\sqrt{m+\frac{1}{2}}\right)^2,\\
      \tilde{\sigma}_{n,m}:=\left(\sqrt{n+\frac{1}{2}}+\sqrt{m+\frac{1}{2}}\right)\left(\left(n+\frac{1}{2}\right)^{-\frac{1}{2}}+\left(m+\frac{1}{2}\right)^{-\frac{1}{2}}\right)^{\frac{1}{3}},
      \end{array}
    \end{equation*}
    \begin{equation*}
              \gamma_{L_N}:=\frac{\tilde{\mu}_{N-1,N+a}\sqrt{\tilde{\sigma}_{N,N+a-1}}}{\tilde{\mu}_{N,N+a-1}\sqrt{\tilde{\sigma}_{N-1,N+a}}},
    \end{equation*}
    \begin{equation*}
        \mu_{L}:=\left(\frac{1}{\sqrt{\tilde{\sigma}_{N-1,N+a}}}+\frac{1}{\sqrt{\tilde{\sigma}_{N,N+a-1}}}\right)\left(\frac{1}{\tilde{\mu}_{N-1,N+a}\sqrt{\tilde{\sigma}_{N-1,N+a}}}+\frac{1}{\tilde{\mu}_{N,N+a-1}\sqrt{\tilde{\sigma}_{N,N+a-1}}}\right)^{-1}
    \end{equation*}
    and
        \begin{equation*}
        \sigma_{L}:=\left(1+\gamma_{L_N}\right)\left(\frac{1}{\tilde{\sigma}_{N-1,N+a}}+\frac{\gamma_{L_N}}{\tilde{\sigma}_{N,N+a-1}}\right)^{-1}.
    \end{equation*}
Let $\mathcal{X}_{L_{a,N}}^{Soft}$ denote the soft-edge-scaled point process for $L_{a,N}$. The kernel is given by $K_{L_{a,N}}^{Soft}(x,y):=\sigma_{L}K_{L_{a,N}}(t_{L_{N}}(x),t_{L_{N}}(y))$ where $t_{L_N}(x):=\mu_{L}+\sigma_{L}x$ for any $x\in I$. Then on $A\subset I$, for any $s\in\mathbb{R}$, there exists an $n(s,\gamma)\in\mathbb{N}$ such that for any $N>n(s,\gamma)$,
\begin{equation}\label{main L soft}
    W_1(\mathcal{N}_{L_{a,N}}^{Soft},\mathcal{N}_{Ai})\le\frac{g(s)e^{-s}}{N^{\frac{2}{3}}}
\end{equation}
for $g(s)$ a continuous non-increasing function of $s$.
\end{theorem}
\begin{theorem}\label{J soft}
    (JUE soft edge) Let $I\subset[s,+\infty)$ for $s>-\infty$ and $J_{a,b,N}$ be a JUE with parameters satisfying $\lim\limits_{N\rightarrow\infty}\frac{a(N)}{N}=\alpha\in(0,\infty)$ and $\lim\limits_{N\rightarrow\infty}\frac{b(N)}{N}=\beta\in[0,\infty)$. Introduce parameters:
    \begin{equation*}
        \cos\phi(N):=\frac{a(N)-b(N)}{a(N)+b(N)+2N+1},
    \end{equation*}
    \begin{equation*}
      \cos\theta(N):=\frac{a(N)+b(N)}{a(N)+b(N)+2N+1},
    \end{equation*}
     \begin{equation*}
        x(N):=-\cos\left(\phi(N)+\theta(N)\right),
    \end{equation*}
     \begin{equation*}
    y^3(N):=\frac{2\sin^2\left(\phi(N)+\theta(N)\right)}{(a(N)+b(N)+2N+1)^2\sin(\phi(N))\sin(\theta(N))}
    \end{equation*}
    and
\begin{equation*}
        u_N:=\tanh^{-1}(x(N)),~v_N:=\frac{y(N)}{1-x^2(N)}.
\end{equation*}
Further define:
\begin{equation*}
    \mu_{J}:=\frac{v_N^{-1}u_N+v_{N-1}^{-1}u_{N-1}^{-1}}{v_N^{-1}+v_{N-1}^{-1}}~and~\sigma_{J}:=2\left(v_{N}^{-1}+v_{N-1}^{-1}\right)^{-1}.
\end{equation*}
Let $\mathcal{X}_{J_{a,b,N}}^{Soft}$ denote the soft-edge-scaled point process for $J_{a,b,N}$. The kernel is given by $K_{J_{a,N}}^{Soft}(x,y):=\sqrt{t'_{J_N}(x)t'_{J_N}(y)}K_{J_{a,b,N}}(t_{J_{N}}(x),t_{J_{N}}(y))$ where $t_{J_N}(x):=\tanh\left(\mu_{J}+\sigma_{J}x\right)$ for any $x\in I$. Then on $A\subset I$, for any $s\in\mathbb{R}$, there exists an $n(s,\gamma)\in\mathbb{N}$ such that for any $N>n(s,\alpha,\beta)$,
\begin{equation}\label{main J soft}
    W_1(\mathcal{N}_{J_{a,b,N}}^{Soft},\mathcal{N}_{Ai})\le\frac{h(s)e^{-\frac{s}{2}}}{N^{\frac{2}{3}}}
\end{equation}
for $h(s)$ a continuous non-increasing function of $s$.
\end{theorem}

Using the soft edge bounds from Theorems \ref{G soft}, \ref{L soft} and \ref{J soft}, we have following results about the largest eigenvalue for different ensembles.
\begin{corollary}\label{largest}
    Let $\lambda_{max}^N$ denote the soft edge-scaled largest eigenvalue of the GUE, LUE or JUE and $F_2(s)$ be the distribution function for the Tracy-Widom~distribution with $\beta=2$. Assume the same parameters stipulations in Theorems \ref{G soft}, \ref{L soft} and \ref{J soft}. For any $s\in\mathbb{R}$, we have
    \begin{equation*}
        |\mathbb{P}\{\lambda_{\max}^N\le s\}-F_2(s)|\le \frac{f(s)}{N^{\frac{2}{3}}}.
    \end{equation*}
for $N$ sufficiently large where $f(s)$ is the same term as what appears in the numerator for (\ref{main G soft}), (\ref{main L soft}) and (\ref{main J soft}).
\end{corollary}
\begin{proof} By the same logic as in the proof of Corollary \ref{smallest},
\begin{equation*}
\begin{split}
     |\mathbb{P}\{\lambda_{\max}^N\le s\}-F_2(s)|&=|\mathbb{P}\{\mathcal{N}^{Soft}((s,\infty))=0\}-\mathbb{P}\{\mathcal{N}_{Ai}((s,\infty))=0\}|
     \\&\le d_{TV}(\mathcal{N}^{Soft}((s,\infty)),\mathcal{N}_{Ai}((s,\infty)))
     \\&\le W_1(\mathcal{N}^{Soft}((s,\infty)),\mathcal{N}_{Ai}((s,\infty)))\\&\le\frac{f(s)}{N^{\frac{2}{3}}}
\end{split}
\end{equation*}
for $N$ large enough, from the previous theorems.
\end{proof}
With the help of Lemma \ref{M and M}, we can get rid of the exponential term derived by comparing the Fredholm determinants as in \cite[Lemma $1.1$]{el2006rate}, \cite[Equation $(32)$]{johnstone2008multivariate} and \cite[Equation $(26)$]{johnstone2012fast}. Consequently, 
Corollary \ref{largest} above actually improved the similar results in those three papers as functions $f$'s appearing in this corollary are smaller than they are in
\cite[Theorem $2$]{el2006rate}, \cite[Theorem $3$]{johnstone2008multivariate} and \cite[Theorem $1$]{johnstone2012fast}.
\subsection{Comparison with Other Results}

Here we list some well-known ROC results in RMT and briefly compare them with our results.

The main known ROC results for the GUE are macroscopic in nature and are proven by completely different mathematical methods. We list several of the major results to contrast with our results. In $2005$ Götze and Tikhomirov, building on a number of works that can be found in that paper, proved the expected spectral distribution function of GUE converges to the semicircle law with the conjectured optimal ROC of $O(N^{-1})$ \cite{gotze_rate_2005}. The ROC was later extended to all Wigner matrices in \cite{gotze_optimal_2016}. In addition to work on the expected spectral distribution, there has been much work on ROCs for the empirical spectral measure for Wigner matrices (including the GUE). We cite an ROC of $O(\frac{\log{N}^C}{N})$ for some constant $C$ for convergence of the empirical spectral measure of Wigner matrices with the semicircle law, but do not go deeper into this line of research \cite{gotze_local_2018}.

There is another line of research looking at convergence of the empirical spectral measure to the semicircle law in $L^p$-Wasserstein metrics. We list \cite{meckes2013concentration} where they obtained $O(N^{-\frac{2}{3}})$ for the empirical spectral measure in $L^1$-Wasserstein metric, as well as \cite{dallaporta_eigenvalue_2012} which obtained $O(\frac{\log{N}}{N^2})$ for the expectation of the square of $L^2$-Wasserstein distance between the spectral measure and the semicircle law.

All the ROC results mentioned above are macroscopic, or global, which is different from the mesoscopic regime. It is also important to emphasize the distinction between our mesoscopic perspective and the microscopic scale in RMT, often associated with eigenvalue rigidity. Microscopic rigidity, such as the local semicircle (see \cite{erdos_rigidity_2012}) or Marchenko-Pastur (see \cite{pillai_universality_2014}) laws, describes the extreme precision with which a single eigenvalue is located near its classical prediction. While our discussion is centered on counting functions, it would be incorrect to view mesoscopic information as stronger or weaker than macro- or microscopic results. Although eigenvalue rigidity is strong enough to derive deviation inequalities for counting functions as \cite[Equation $(9)$]{dallaporta_eigenvalue_2012} and \cite[Equation $(10)$]{dallaporta2013eigenvaluevarianceboundscovariance}, it does not directly imply our specific mesoscopic ROCs. This is because a direct comparison with our target processes (sine, Airy, Bessel) is not meaningful, as these limiting point processes themselves lack such rigidity. Therefore, our mesoscopic discussion should be regarded as a line of inquiry that runs parallel to the classical global and local results in RMT, without a strict hierarchical relationship of inclusion.

 Our findings at the mesoscopic scale also carry significant weight due to the flexibility of our framework, which allows the spatial regime to vary with the overall system size. Unlike a microscopic analysis, which is confined to a fixed, local scale, our approach permits the scaling parameter $s$ to change, subject to asymptotic conditions that ensure we remain within the mesoscopic regime. This dynamic scaling indeed covers some the rigid setting of microscopic limits, for example, Forrester's discussion for the convergence of the kernel functions in \cite{log-gas}, and provides a powerful tool for probing the spatial structure of the point process across different scales. Indeed, aside from extreme values of the parameter, we cannot expect our results to be optimal in the spatial parameter $s$, as the estimates used in the proof are not sharp. However, for any fixed $s$, we conjecture that our results—which in fact cover certain microscopic regimes—are optimal with respect to the dimension $N$.
    
The microscopic ROC for GUE bulk has been indicated in the Chapter $9$ of \cite{log-gas}. The author proved $O(N^{-\frac{1}{2}})$ pointwise convergence for the scaled kernel function to the sine kernel function in $(9.5)$, which already implies $O(N^{-\frac{1}{2}})$ as the mesoscopic ROC for the GUE towards sine process with respect to $W_1$ as in our discussion. Our result in fact improved this rate to $O(N^{-1})$, which has the same order as the macroscopic ROC in \cite{gotze_rate_2005} mentioned above. Therefore, it is reasonable to conjecture our mesoscopic rate for GUE bulk is optimal.

For the GUE, LUE, and JUE soft edge we cite three papers: \cite{el2006rate}, \cite{johnstone2008multivariate}, and \cite{johnstone2012fast}. Following the seminal work of \cite{tracy1994level} and \cite{tracy1996orthogonal} on the Tracy-Widom distribution there was considerable interest in proving ROC results for the largest eigenvalues of various random matrix ensembles. Without going into full details we highlight that each of the three above papers proved $O(N^{-\frac{2}{3}})$ convergence for the largest eigenvalues in the GUE, LUE, and JUE to the Tracy-Widom distribution. Crucially for this work, in all three cases the authors had to make the precise trace class estimates we require for our results (albeit for a related but different reason). As highlighted above, our soft edge results immediately imply the ROCs obtained in these papers. The method in all of those three papers relies heavily on the auxiliary variable introduced by the Liouville-Green transform. Follow the sophisticated analysis in \cite{el2006rate}, we see it is unrealistic to expect a rate higher than $N^{\frac{2}{3}}$, as the power $2/3$ is intrinsic in our L-G transform around the largest eigenvalue. As a result, we guess, it is very likely that the order $N^{-\frac{2}{3}}$ is optimal in all three cases.

For the LUE hard edge, Forrester's discussion in \cite{log-gas} indeed indicates $O(N^{-1})$ as the mesoscopic ROC for the scaled LUE towards the Bessel process. He proved his conjecture that, pointwise, the ROC for the kernel function can be improved to $O(N^{-2})$ after some better rescaling in \cite{forrester2019finite}. Our result in Theorem \ref{L hard} implies that not only is $O(N^{-2})$ the pointwise ROC for the kernel function, but also the mesoscopic ROC for the point process with respect to the $W_1$ metric, which strengthens Forrester's result. As our result matches the optimal pointwise ROC of the kernel function, we speculate the order $N^{-2}$ is also optimal in our mesoscopic regime.

Another important non-Hermitian ensemble in RMT is the ensemble on some classical compact group. For example, the joint distribution of the eigenvalues for a Haar matrix on the $N\times N$ unitary group $\mathbb{U}_N$, which is also called a {\bf Circular~Unitary~Ensemble~(CUE)}. A comparable messocropic ROC for CUE to the sine point process was given as $O(N^{-\frac{3}{2}})$ by E. Meckes and M. Meckes in \cite{meckes2015self}. An improvement of that result and similar results for the other compact classical groups can be seen in \cite{cai2025ratesbulkconvergenceensembles}.
\section{Proofs}\label{Proof}
As already stated, our main challenge is to control the trace class norm $\|\cdot\|_1$ of the difference between two operators. However, directly computing the trace class norm of the difference between two operators is not easy. 

For an integral operator, the Hilbert-Schmidt norm (the Schatten $2$-norm) is easier to calculate than the trace class norm and is given by 

\begin{equation}\label{2 to L^2}   \|\mathcal{K}\|_{2,U}:=\left(\sum_{s_j(\mathcal{K})\in U}s_j^2(\mathcal{K})\right)^{\frac{1}{2}}=\left(\int_U\int_U |K|^2(x,y)dydx\right)^{\frac{1}{2}}=\|K\|_{L^2(U^2)}.
\end{equation}

So we use the \textbf{operator~Cauchy-Schwarz~inequality}:
\begin{equation}\label{CS in}
\|\mathcal{K}_1\mathcal{K}_2\|_1\le\|\mathcal{K}_1\|_2\|\mathcal{K}_2\|_2=\|K_1\|_{L^2}\|K_2\|_{L^2}
\end{equation}
where $\mathcal{K}_1$ and $\mathcal{K}_2$ are trace class with kernel $K_1$ and $K_2$ respectively to go from estimating trace class norms to estimating Hilbert-Schmidt norms (see \cite{simon2015operator} as a good reference for the operator theory basics we list here). 

The work in the proofs boils down to two things. First, decomposing the operators appropriately so that Cauchy-Schwarz can be used. Second, getting non-asymptotic approximations of classical special functions needed to finish the estimates. Our methods are similar to \cite{el2006rate} and \cite{johnstone2012fast} where they made the trace class estimates that we use for the soft edge results. 

The following inequality is useful in this paper.
\begin{lemma}\label{21in}
    If $\mathcal{A}$, $\mathcal{B}$ and $\mathcal{C}$ are Hilbert-Schmidt operators on a Hilbert space $X$, then the following inequality holds:
    \begin{equation*}
2\|\mathcal{A}\mathcal{B}+\mathcal{B}\mathcal{A}-\mathcal{C}\mathcal{C}\|_1\le\|\mathcal{A}+\mathcal{B}-\sqrt2\mathcal{C}\|_2\|\mathcal{A}+\mathcal{B}+\sqrt2\mathcal{C}\|_2+\|\mathcal{A}-\mathcal{B}\|^2_2.
    \end{equation*}
\end{lemma}
\begin{proof} By the Cauchy-Schwarz inequality (\ref{CS in}), we see that $\|\mathcal{A}^2-\mathcal{B}^2\|_1=\frac{1}{2}\|(\mathcal{A}+\mathcal{B})(\mathcal{A}-\mathcal{B})+(\mathcal{A}-\mathcal{B})(\mathcal{A}+\mathcal{B})\|_1\le\|\mathcal{A}-\mathcal{B}\|_2\|\mathcal{A}+\mathcal{B}\|_2$ and hence,
\begin{align*}
  2\|\mathcal{A}\mathcal{B}+\mathcal{B}\mathcal{A}-\mathcal{C}^2\|_1&=\|(\mathcal{A}+\mathcal{B})^2-(\mathcal{A}-\mathcal{B})^2-2\mathcal{C}^2\|_1\\&\le\|(\mathcal{A}+\mathcal{B})^2-2\mathcal{C}^2\|_1+\|(\mathcal{A}-\mathcal{B})^2\|_1\\&\le\|\mathcal{A}+\mathcal{B}-\sqrt2\mathcal{C}\|_2\|\mathcal{A}+\mathcal{B}+\sqrt2\mathcal{C}\|_2+\|\mathcal{A}-\mathcal{B}\|_2^2.&\qedhere
\end{align*}
\end{proof}

 \subsection{GUE bulk}\label{GUE bulk}
We first show how to decompose the operators, then give the function approximation needed, and finally put these together to give the desired result for this case.

Our factorization relies on the philosophy of the following lemma.
\begin{lemma}\label{FA}
    Suppose $\mathcal{K}$ is an integral operator with kernel $K(x,y)$ defined on $D^2\subset\mathbb{R}^2$ for some domain $D$. If $K$ can be expressed as
    \begin{equation*}
        K(x,y)=\int_D m(xt)n(yt)dt
    \end{equation*}
    for some continuous function $m,~n$ on $D$, then
        $\mathcal{K}=\mathcal{M}\mathcal{N}$
    where $\mathcal{M}$ and $\mathcal{N}$ have the kernels $m(xy)$ and $n(xy)$ respectively.
\end{lemma}
   \begin{proof} For some suitable function $f$ defined on $D$, by Fubini's theorem, we have
\begin{align*}
        (\mathcal{M}\mathcal{N})(f)(x)&=\int_D m(xy)(\mathcal{N}f)(y)dy=\int_Dm(xy)\left(\int_Dn(yz)f(z)dz\right)dy\\
        &=\int_D\int_Df(z)m(xy)n(zy)dzdy\\
        &=\int_Df(z)dz\int_Dm(xy)n(zy)dy\\
        &=\int_DK(x,z)f(z)dz=\mathcal{K}f(x). &\qedhere
\end{align*}
\end{proof}

\begin{proposition}\label{GUE fac}
	{\rm(GUE decomposition)} Let $I\coloneqq [-s,s]$ for $0 < s < \infty$.
	Define integral operators $\mathcal{C}$ and $\mathcal{S}$ on $I^2$ with kernels 
	\begin{equation*}
	C(x,y) = \frac{1}{\sqrt{2s}}\cos\left(\frac{\pi}{s} xy\right), \qquad S(x,y) = \frac{1}{\sqrt{2s}}\sin\left(\frac{\pi}{s} xy\right).
	\end{equation*} 
	Then, 
	\begin{equation}\label{sine fac}
	\mathcal{K}_{Sine} = \mathcal{C}\mathcal{C} + \mathcal{S}\mathcal{S}.
	\end{equation}
	Let $c_N = \frac{\pi}{\sqrt{2N}}$, and define the following integral operators $\mathcal{E}_{N}$, $\mathcal{E}_{N-1}$, $\mathcal{F}_{N}$, $\mathcal{F}_{N-1}$, $\mathcal{G}_{N}$, and $\mathcal{G}_{N-1}$ with kernels on $I^2$ given by
	\begin{equation*}
	\begin{split}
	E_{N}(x,y) & = \left(\frac{\pi}{2s}\right)^{\frac{1}{2}}
	\left(\frac{N}{2}\right)^{\frac{1}{4}}
	\text{ } \psi_N\left(\frac{c_N}{s}xy\right), \\
	E_{N-1}(x,y) & = \left(\frac{\pi}{2s}\right)^{\frac{1}{2}}
	\left(\frac{N}{2}\right)^{\frac{1}{4}} 
	\text{ } \psi_{N-1}\left(\frac{c_N}{s}xy\right),\\
	F_{N}(x,y) & = \left( \frac{\pi}{2s}\right)	\left(\frac{1}{2N}\right)^{\frac{1}{4}}
	\text{ } \left(xy\right) \text{ } \psi_{N}\left(\frac{c_N}{s}xy\right),\\
	F_{N-1}(x,y) & = \left( \frac{\pi}{2s}\right) \left(\frac{1}{2N}\right)^{\frac{1}{4}}
	\text{ } (xy) \text{ }\psi_{N-1}\left(\frac{c_N}{s}xy\right),\\
	G_{N}(x,y) & = \left( \frac{\pi}{2s}\right)	\left(\frac{1}{2N}\right)^{\frac{1}{4}} 
	\text{ } \psi_{N}\left(\frac{c_N}{s}xy\right), \text{ and }\\
	G_{N-1}(x,y) & = \left( \frac{\pi}{2s}\right)	\left(\frac{1}{2N}\right)^{\frac{1}{4}} 
	\text{ } \psi_{N-1}\left(\frac{c_N}{s}xy\right)
	\end{split}
	\end{equation*}
where $\psi_j(x):=\frac{e^{-\frac{x^2}{2}}}{\sqrt{\sqrt\pi 2^j j!}}H_j(x)=\frac{1}{\sqrt{\sqrt\pi 2^j j!}}e^{\frac{x^2}{2}}\frac{d^j}{x^j}e^{-x^2}$ is the $j$th scaled Hermite function. 
	Then,
	\begin{equation}\label{G fac}
	\begin{split}
	\mathcal{K}^{Bulk}_{G_N}  
	= \mathcal{E}_{N}\mathcal{E}_{N} + \mathcal{E}_{N-1}\mathcal{E}_{N-1}
	+ \mathcal{F}_{N}\mathcal{G}_{N-1} + \mathcal{F}_{N-1}\mathcal{G}_{N}
	+ \mathcal{G}_{N-1}\mathcal{F}_{N} + \mathcal{G}_{N}\mathcal{F}_{N-1}.
	\end{split}
	\end{equation}
\end{proposition}
\begin{proof} First of all, define $N_{Sine}(x,y):=\sin(\pi(x-y))$, which is the numerator of the sine kernel in \eqref{sinek}. We have
\begin{equation}\label{int4H}
    N_{Sine}(x,y)=\frac{1}{2}\left(N_{Sine}(x,y)-N_{Sine}(-x,-y)\right)=\frac{1}{2}\int_{-1}^1\frac{dN_{Sine}(tx,ty)}{dt}dt.
\end{equation}
Taking the derivative of $N_{Sine}(tx,ty)$ with respect to $t$ directly gives us
	\begin{equation*}
	\begin{split}
	K_{Sine}(x,y) 
	& = \frac{1}{2} \int_{-1}^1 \left[\cos{(\pi xt)}\cos{(\pi yt)} + \sin{(\pi xt)}\sin{(\pi yt)} \right]dt, \\
	\end{split}
	\end{equation*}
	and changing variables
	\begin{equation}
	\label{eq:Sine_factored}
	\begin{split}
	K_{Sine}(x,y) 
	& = \frac{1}{2s} \int_{-s}^s \left[\cos{\bigg(\frac{\pi}{s} xt\bigg)}\cos{\bigg(\frac{\pi}{s}yt\bigg)} 
	+ \sin{\bigg(\frac{\pi}{s}xt\bigg)}\sin{\bigg(\frac{\pi}{s}yt\bigg)} \right] dt. \\
	\end{split}
	\end{equation}

Secondly, define
	\begin{equation*}
	N_{G_N}(x,y) 
	:= \psi_{N}(c_N x)\psi_{N-1}(c_N y) - \psi_{N-1}(c_N x) \psi_{N}(c_N y).
	\end{equation*}
    The fact that $K_{G_N}^{Bulk}(x,y)=\frac{\sqrt{N/2}}{x-y}N_{G_N}(x,y)$ follows by combining \eqref{kernel1} and \eqref{kernel2}.
In addition, as one of $\psi_N$ and $\psi_{N-1}$ is odd and the other is even, their product is odd giving $-N_{G_N}(x,y) = N_{C_N}(-x,-y)$. Hence,
	\begin{equation}
	N_{G_N}(x,y)= \frac{1}{2} \int_{-1}^1 \frac{d}{dt} N_{G_N}(xt,yt) dt. \\
	\end{equation}
	
By the recurrence $5.5.10$ in \cite{szego1975orthogonal}:
	\begin{equation*}
	\psi_N'(x) = \sqrt{2N} \psi_{N-1}(x) - x \psi_N(x),~\psi_{N-1}'(x) = x \psi_{N-1}(x) - \sqrt{2N}\psi_N(x),
	\end{equation*}
it is not difficult to get
	\begin{equation*}
	\begin{split}
	\frac{d}{dt} & \bigg[N_{G_N}(xt,yt)\bigg] \\ 
	= & c_N \bigg(
	(x-y)\sqrt{2N}\big[ 
	\psi_{N-1}(c_N x t) \psi_{N-1}(c_N y t) + \psi_{N}(c_N x t) \psi_{N}(c_N y t)\big] \\
	& - c_N t (x-y)(x+y) \big[ 
	\psi_{N}(c_N x t) \psi_{N-1}(c_N y t) + \psi_{N-1}(c_N x t) \psi_{N}(c_N y t)\big]
	\bigg).
	\end{split}
	\end{equation*}
	Plugging this into the integral of \eqref{int4H} gives
	\begin{equation*}
	\begin{split}
	K_{G_N}^{Bulk}(x,y) 
	& = \sqrt{\frac{N}{2}}\frac{1}{2(x-y)} \int_{-1}^1 \frac{d}{dt} \bigg[N_{G_N}(xt,yt)\bigg] dt \\
	= & \frac{\pi}{2} \sqrt{\frac{N}{2}} \int_{-1}^1 \big[ 
	\psi_{N-1}(c_N x t) \psi_{N-1}(c_N y t) + \psi_{N}(c_N x t) \psi_{N}(c_N y t) \big] dt \\
	& - \frac{\pi^2}{2^{\frac{5}{2}}N^{\frac{1}{2}}} 
	\int_{-1}^1(xt+yt) \big[ 
	\psi_{N}(c_N x t) \psi_{N-1}(c_N y t) + \psi_{N-1}(c_N x t) \psi_{N}(c_N y t) \big] dt. \\
   = & \frac{\pi}{2} \sqrt{\frac{N}{2}} \int_{-1}^1 \psi_{N-1}(c_N x t) \psi_{N-1}(c_N y t) dt 
	+ \frac{\pi}{2} \sqrt{\frac{N}{2}} \int_{-1}^1 \psi_{N}(c_N x t) \psi_{N}(c_N y t) dt \\
	& - \frac{\pi^2}{2^{\frac{5}{2}}N^{\frac{1}{2}}} 
	\int_{-1}^1 \big[ xt \psi_{N}(c_N x t) \big] \psi_{N-1}(c_N y t) dt  
	- \frac{\pi^2}{2^{\frac{5}{2}}N^{\frac{1}{2}}} 
	\int_{-1}^1 \big[ xt \psi_{N-1}(c_N x t) \big] \psi_{N}(c_N y t) dt \\  
	& - \frac{\pi^2}{2^{\frac{5}{2}}N^{\frac{1}{2}}}  
	\int_{-1}^1 \psi_{N-1}(c_N x t) \big[yt \psi_{N}(c_N y t) \big] dt 
	- \frac{\pi^2}{2^{\frac{5}{2}}N^{\frac{1}{2}}}  
	\int_{-1}^1 \psi_{N}(c_N x t) \big[yt \psi_{N-1}(c_N y t) \big] dt.
	\end{split}
	\end{equation*}
	Changing variables completes the proof.\end{proof}
 
 The following proposition quantifies how well Hermite functions approximate cosines and provides us with a useful bound for $\|\cdot\|_2$ after decomposition in \eqref{sine fac} and \eqref{G fac}.
\begin{proposition}\label{H2C}
	Let $\psi_N(x)$ be the $N$th Hermite function and let $N\geq 2$. Then,
	\begin{equation}
	\label{HC in}
    \begin{array}{ll}
	\left\lvert \pi^{\frac{1}{2}} \left(\frac{N}{2}\right)^{\frac{1}{4}} \psi_N\left(\frac{x}{\sqrt{2N}}\right) - \cos{\left(x-\frac{N\pi}{2}\right)} \right\rvert
	\leq \frac{p_3(\lvert x \rvert)}{N}
    \\
    \\
	\left\lvert \pi^{\frac{1}{2}} \left(\frac{N}{2}\right)^{\frac{1}{4}} \psi_{N-1}\left(\frac{x}{\sqrt{2N}}\right) - \cos{\left(x- \frac{(N-1)\pi}{2}\right)} \right\rvert
	\leq \frac{p_3(\lvert x \rvert)}{N},
    \end{array}
	\end{equation}
	where $p_3$ is a degree $3$ polynomial with non-negative coefficients.
\end{proposition}
\begin{proof}
	As noted in Section $8.65$ of \cite{szego1975orthogonal}, the unscaled Hermite function has the following expansion:
	\begin{equation}
	\label{eq:Szego_Hermite1}
	\begin{split}
	H_N(x) e^{-\frac{x^2}{2}} &
	= \lambda_N  \cos{\left(x\sqrt{2N+1}-\frac{N\pi}{2}\right)} \\
	&~+	\frac{1}{\sqrt{2N+1}} \int_0^x \sin{\left(\sqrt{2N+1}(x-t)\right)}t^2 H_N(t) e^{-\frac{t^2}{2}}dt,
	\end{split}
	\end{equation}
	for
	\[
	\lambda_N=
	\begin{cases}  
	\frac{(N+1)!}{(\frac{N+1}{2})!}\frac{1}{\sqrt{2N+1}}, & N \text{ odd} \\
	\frac{N!}{\left(\frac{N}{2}\right)!}, & N \text{ even}. 
	\end{cases}
	\]
	Introduce parameters
	\begin{equation*}
	d_N \coloneqq \pi^{\frac{1}{2}}\left(\frac{N}{2}\right)^{\frac{1}{4}} 
	\qquad \text{ and } \qquad
	e_N \coloneqq \frac{1}{\pi^{\frac{1}{4}}2^{\frac{N}{2}}(N!)^{\frac{1}{2}}}.	
	\end{equation*}
	Then, we have 
	\begin{equation}
	\label{ineq:Hermite_cosine_calc_1}
	\begin{split}
	&\bigg\lvert \pi^{\frac{1}{2}}\left(\frac{N}{2}\right)^{\frac{1}{4}}\psi_N\left(\frac{x}{\sqrt{2N}}\right)  -\cos{\left(x-\frac{N\pi}{2}\right)} \bigg\rvert  \\
	&\leq
	\left\lvert d_N e_N \lambda_N \cos{\left(\sqrt{\frac{2N+1}{2N}}x-\frac{N\pi}{2}\right)} - \cos{\left(x-\frac{N\pi}{2}\right)} \right\rvert \\
	&~+ \left\lvert \frac{d_N e_N}{\sqrt{2N+1}} \int_0^{\frac{x}{\sqrt{2N}}} \sin{\left(\sqrt{2N+1}(x(2N)^{-1/2}-t)\right)}t^2 H_N(t) e^{-\frac{t^2}{2}}dt\right\rvert. \\
	\end{split}
	\end{equation}	
	For the second term, notice $e_N$ is the scale of the Hermite function $\psi_N$, and $\lvert \psi_N(x)\rvert \leq
	\pi^{-\frac{1}{4}}$ for all $N$ and $x$ (see \cite[Equation $(2)$]{indritz1961inequality}). We find  
	\begin{equation}
	\label{eq:lemma2}
	\begin{split}
	\bigg\lvert \frac{d_N}{\sqrt{2N+1}} & \int_0^{\frac{x}{\sqrt{2N}}} 
	\sin{\left(\sqrt{2N+1}(\frac{x}{\sqrt{2N}}-t)\right)}t^2 \psi_N(t) dt\bigg\rvert \\ 
	& \leq  \frac{\pi^{\frac{1}{2}}N^{\frac{1}{4}}}{2^{\frac{1}{4}}(2N+1)^{\frac{1}{2}}}
	\int_0^{\frac{\lvert x\rvert}{\sqrt{2N}}} \left\lvert \sin{\left(\sqrt{2N+1}(\frac{x}{\sqrt{2N}}-t)\right)}\right\rvert 
	\left\lvert \psi_N(t)\right\rvert t^2 dt \\
	& \leq \frac{\pi^{\frac{1}{4}}}{2^{\frac{3}{4}}N^{\frac{1}{4}}}  
	\int_0^{\frac{\lvert x\rvert}{\sqrt{2N}}} t^2 dt \le \frac{\lvert x\rvert^3}{N^{\frac{7}{4}}}.
	\end{split}
	\end{equation}	
	Now deal with the first term in \eqref{ineq:Hermite_cosine_calc_1}. From the inequality 
	\begin{equation}
	\label{ineq:Hermite_cosine_2}
	\begin{split}
	\bigg\lvert d_N e_N & \lambda_N \cos{\left(\sqrt{\frac{2N+1}{2N}}x-\frac{N\pi}{2}\right)} 
	- \cos{\left(x-\frac{N\pi}{2}\right)} \bigg\rvert \\
	\leq & \left\lvert d_N e_N \lambda_N \cos{\left(\sqrt{\frac{2N+1}{2N}}x-\frac{N\pi}{2}\right)} 
	- d_N e_N \lambda_N \cos{\left(x-\frac{N\pi}{2}\right)} \right\rvert \\
	& + \bigg\lvert d_N e_N \lambda_N  \cos{\left(x-\frac{N\pi}{2}\right)}
	- \cos{\left(x-\frac{N\pi}{2}\right)} \bigg\rvert \\
	= & \bigg\lvert d_N e_N \lambda_N \bigg\rvert 
	\bigg\lvert \cos{\left(\sqrt{\frac{2N+1}{2N}}x-\frac{N\pi}{2}\right)} 
	- \cos{\left(x-\frac{N\pi}{2}\right)} \bigg\rvert \\
	& + \bigg\lvert d_N e_N\lambda_N - 1 \bigg\rvert 
	\bigg\lvert \cos{\left(x-\frac{N\pi}{2}\right)} \bigg\rvert \\
	\leq & \bigg\lvert d_N e_N \lambda_N \bigg\rvert 
	\bigg\lvert \left(\sqrt{1 + \frac{1}{2N}} - 1\right) x \bigg\rvert
	+ \bigg\lvert d_N e_N\lambda_N - 1 \bigg\rvert \\
	\leq & \big\lvert d_N e_N \lambda_N \big\rvert \frac{\lvert x\rvert}{2N}
	+ \big\lvert d_N e_N\lambda_N - 1 \big\rvert, \\
	\end{split}
	\end{equation}	
	we see $d_N e_N \lambda_N$ actually provides us with a proper order we want. According to the inequality $(2)$ in \cite{robbins1955remark}, we have a sharp form of Stirling's approximation:
	\begin{equation}\label{fto}
	\sqrt{2\pi} n^{n+\frac{1}{2}}e^{-n}e^{\frac{1}{12n+1}} \leq n! \leq
	\sqrt{2\pi} n^{n+\frac{1}{2}}e^{-n}e^{\frac{1}{12n}}. 
	\end{equation}
	Then, for $N$ even,
	\begin{equation}
	d_N e_N \lambda_N 
	= \left(\frac{\pi N}{2}\right)^{\frac{1}{4}} \frac{1}{2^{\frac{N}{2}}} 
	\frac{\left(N!\right)^{\frac{1}{2}}}{\left(\frac{N}{2}\right)!}.
	\end{equation} 
	Using the bounds in \eqref{fto}, we have
	\begin{equation*}
	\exp{\left(\frac{1}{2(12N+1)} - \frac{1}{6N}\right)} \leq d_N e_N \lambda_N \leq \exp{\left(\frac{1}{24N} - \frac{1}{6N+1}\right)}.
	\end{equation*}	
	Clearly $1$ is an upper bound on the right, and, for the lower bound, we just apply the fact that $e^x\geq 1+x$.  We obtain
	\begin{equation*}
	1 - \frac{1}{N} \leq d_N e_N \lambda_N \leq 1.
	\end{equation*}		
	For $N$ odd, we have 
	\begin{equation}
	d_N e_N \lambda_N 
	= \left(\frac{\pi N}{2}\right)^{\frac{1}{4}} \frac{1}{2^{\frac{N}{2}}} 
	\frac{(N+1)^{\frac{1}{2}}}{(2N+1)^{\frac{1}{2}}}
	\frac{\left((N+1)!\right)^{\frac{1}{2}}}{\left(\frac{N+1}{2}\right)!}.
	\end{equation} 
	Similarly, bounds in \eqref{fto} yield
	\begin{equation*}
	d_N e_N \lambda_N 
	\leq \left(\frac{N}{N+1}\right)^{\frac{1}{4}} \left(\frac{(2N+1)+1}{2N+1}\right)^{\frac{1}{2}} 
	\exp{\left[\frac{1}{2(12(N+1))} - \frac{1}{6(N+1)+1}\right]}
	\end{equation*}	
	and 
	\begin{equation*}
	d_N e_N \lambda_N 
	\geq \left(\frac{N}{N+1}\right)^{\frac{1}{4}} \left(\frac{(2N+1)+1}{2N+1}\right)^{\frac{1}{2}} 
	\exp{\left[\frac{1}{2(12(N+1))+1} - \frac{1}{6(N+1)}\right]}.
	\end{equation*}
	For $N\geq 2$ the terms not involving the exponentials are bounded above and below by $1\pm\frac{1}{N}$ (use the fact that $\sqrt{1-N^{-1}}\ge1-N^{-1}$) and the exponentials are again bounded by $1$ and $1-\frac{1}{N}$, giving 
	\begin{equation*}
	1 - \frac{2}{N} \leq d_N e_N \lambda_N \leq 1 + \frac{1}{N}. 
	\end{equation*}
	Plugging back into \eqref{ineq:Hermite_cosine_calc_1} we get
	\begin{equation}
	\label{ineq:Hermite_cosine_calc_1_odd}
	\begin{split}
	&\bigg\lvert \pi^{\frac{1}{2}} \left(\frac{N}{2}\right)^{\frac{1}{4}}\psi_N\left(\frac{x}{\sqrt{2N}}\right)  
	 -\cos{\left(x-\frac{N}{2}\pi\right)} \bigg\rvert\\
	 &\leq \lvert d_N c_N \lambda_N \rvert \frac{\lvert x\rvert}{2N} + \lvert d_N c_N \lambda_N -1 \rvert + \frac{\lvert x\rvert^3}{N^{\frac{7}{4}}} \\
	& \leq \bigg(1 + \frac{2}{N}\bigg) \frac{\lvert x\rvert}{2N} + \frac{2}{N} + \frac{\lvert x\rvert^3}{N^{\frac{7}{4}}}\leq \frac{p_3(\lvert x\rvert)}{N},
	\end{split}
	\end{equation} 
	for $p_3$ some degree $3$ polynomial with non-negative coefficients. \end{proof}

\medbreak\noindent {\itshape Proof of Theorem \ref{G bulk}.}\enspace
Recall all operators are restricted on $I^2=[-s,s]^2$. Using Proposition \ref{GUE fac} and the Cauchy-Schwarz inequality, we have 
	\begin{equation}
	\begin{split}
	&\|  \mathcal{K}_{G_N}^{Bulk}   - \mathcal{K}_{Sine}\|_1 \\
	& =\| \mathcal{E}_{N}\mathcal{E}_{N} + \mathcal{E}_{N-1}\mathcal{E}_{N-1} 
	+ \mathcal{F}_{N}\mathcal{G}_{N-1} + \mathcal{F}_{N-1}\mathcal{G}_{N}
	+ \mathcal{G}_{N-1}\mathcal{F}_{N} + \mathcal{G}_{N}\mathcal{F}_{N-1}
	- (\mathcal{C}\mathcal{C} + \mathcal{S}\mathcal{S})\|_1. \\
    &\le\|\mathcal{E}_N\mathcal{E}_N+\mathcal{E}_{N-1}\mathcal{E}_{N-1}-\mathcal{C}\mathcal{C}-\mathcal{S}\mathcal{S}\|_1+2\|\mathcal{F}_N\|_2\|\mathcal{G}_{N-1}\|_2+2\|\mathcal{F}_{N-1}\|_2\|\mathcal{G}_N\|_2
	\end{split}
	\end{equation}
    Let's first consider $\mathcal{F}$ and $\mathcal{G}$. In fact, by Proposition \ref{H2C}, we see
	\begin{equation*}
	\begin{split}
	\lVert \mathcal{F}_N \rVert_2^2 
	= & \int_{-s}^{s} \int_{-s}^{s} 
	\left\lvert\left( \frac{\pi}{2s}\right)	\left(\frac{1}{2N}\right)^{\frac{1}{4}}
	\text{ } \left(xy\right) \text{ } \psi_{N}\left(\frac{c_N}{s}xy\right) \right\rvert^2  
	dx dy \\
	\leq & \frac{\pi s^2}{4N} \int_{-s}^{s} \int_{-s}^{s} \text{ } 
	\left\lvert \pi^{\frac{1}{2}} \frac{N^{\frac{1}{4}}}{2^{\frac{1}{4}}}\psi_{N}\left(\frac{\pi xy}{s\sqrt{2N}}\right) \right\rvert^2  
	dx dy \\	
	\leq & \frac{\pi s^2}{4N} \int_{-s}^{s} \int_{-s}^{s} \text{ } 
	\left[ \text{ }\bigg\lvert \cos{\bigg(\frac{\pi xy}{s}-\frac{N\pi}{2}\bigg)} \bigg\rvert 
	+ \bigg\lvert\frac{p_3\left(\left|\frac{\pi x y}{s}\right|\right)}{N} \bigg\rvert \text{ }\right]^2  
	dx dy \\
	\leq&  \frac{p_4(s)}{N}+\frac{p_7(s)}{N^2}+\frac{p_{10}(s)}{N^3},
	\end{split}
	\end{equation*}
	where $p_4$, $p_7$ and $p_{10}$ are again some polynomials with non-negative coefficients. Taking the square root gives 
	\begin{equation*}
	\begin{split}
	\lVert \mathcal{F}_N \rVert_2
	\leq &  \frac{\sqrt{p_{10}(s)}}{N^\frac{1}{2}}, \\
	\end{split}
	\end{equation*}
	for the same condition on $p_{10}(s)$ just mentioned. 
	A similar calculation shows 
	\begin{equation*}
		\lVert \mathcal{G}_{N-1} \rVert_2 \leq \frac{\sqrt{p_{10}(s)}}{N^\frac{1}{2}}.
	\end{equation*}
	So, 
	\begin{equation*}
	\begin{split}
	\lVert \mathcal{F}_N \rVert_2 \lVert \mathcal{G}_{N-1} \rVert_2
	\leq \frac{p_{10}(s)}{N}.
	\end{split}
	\end{equation*}
	$\|\mathcal{F}_{N-1}\|_2\|\mathcal{G}_N\|_2$ can be handled in the completely same way. Consequently, 
    \begin{equation*}
        \|\mathcal{K}_{G_N}^{Bulk}-\mathcal{K}_{Sine}\|_1\le\|\mathcal{E}_N\mathcal{E}_N+\mathcal{E}_{N-1}\mathcal{E}_{N-1}-\mathcal{C}\mathcal{C}-\mathcal{S}\mathcal{S}\|_1+\frac{p_{10}(s)}{N}
    \end{equation*}
    
	For the first term on right side of the previous inequality, as the Hermite functions have the same parity properties as cosine, we can pair them in different ways. If $N$ is even, we want to group $\mathcal{E}_{N}\mathcal{E}_{N}$ with $\mathcal{C}\mathcal{C}$ and $\mathcal{E}_{N-1}\mathcal{E}_{N-1}$ with $\mathcal{S}\mathcal{S}$. If $N$ is odd, then we group the other way. Everything works identically in either case, so we assume without loss of generality that $N$ is even. Moreover, recall in Proposition \ref{H2C}, the function $E_N(x)$ is approximated by $\cos\left(x-\frac{N\pi}{2}\right)$ instead of $\cos(x)$, the kernel function directly related to $\mathcal{C}$. So, we need to further group operators with respect to the parity of $\frac{N}{2}$. For $\frac{N}{2}$ even, as $\cos\left(x-\frac{N\pi}{2}\right)=\cos(x)$, we can directly group $\mathcal{E}_N$ with $\mathcal{C}$ and $\mathcal{E}_{N-1}$ with $\mathcal{S}$, and hence get
	\begin{equation}\label{GUE in1}
	\|(\mathcal{E}_N\mathcal{E}_N-\mathcal{C}\mathcal{C})+(\mathcal{E}_{N-1}\mathcal{E}_{N-1}-\mathcal{S}\mathcal{S})\|_1\le\|\mathcal{E}_N-\mathcal{C}\|_2(\|\mathcal{E}_N\|_2+\|\mathcal{C}\|_2)+\|\mathcal{E}_{N-1}-\mathcal{S}\|_2(\|\mathcal{E}_{N-1}\|_2+\|\mathcal{S}\|_2)
	\end{equation}
according to Lemma \ref{21in}. When $\frac{N}{2}$ is odd, we need to group $\mathcal{E}_N$ with $-\mathcal{C}$ and $\mathcal{E}_N$ with $-\mathcal{S}$ as $\cos\left(x-\frac{N\pi}{2}\right)=-\cos(x)$ now. Notice $\mathcal{C}\mathcal{C}=(-\mathcal{C})(-\mathcal{C})$ and $\mathcal{S}\mathcal{S}=(-\mathcal{S})(-\mathcal{S})$, by Lemma \ref{21in} again, we have
	\begin{equation}\label{GUE in2}
	\|(\mathcal{E}_N\mathcal{E}_N-\mathcal{C}\mathcal{C})+(\mathcal{E}_{N-1}\mathcal{E}_{N-1}-\mathcal{S}\mathcal{S})\|_1\le\|\mathcal{E}_N+\mathcal{C}\|_2(\|\mathcal{E}_N\|_2+\|\mathcal{C}\|_2)+\|\mathcal{E}_{N-1}+\mathcal{S}\|_2(\|\mathcal{E}_{N-1}\|_2+\|\mathcal{S}\|_2).
	\end{equation}

Now suppose $\frac{N}{2}$ is odd, from \eqref{GUE in2} and Proposition \ref{H2C}, we get
	\begin{equation*}
	\begin{split}
	\big\lVert \mathcal{E}_{N} + \mathcal{C} \big\lVert_2^2 
	= & \int_{-s}^s\int_{-s}^s \left\lvert E_N(x,y) + C(x,y) \right\rvert^2 dx dy \\
	= & \frac{s}{2}\int_{-1}^1\int_{-1}^1 \left\lvert \pi^{\frac{1}{2}} \left(\frac{N}{2}\right)^{\frac{1}{4}} \text{ } \psi_N\left(\frac{s\pi xy}{\sqrt{2N}}\right) - \cos\left(s \pi xy -\frac{N\pi}{2}\right) \right\rvert^2 dx dy \\
    \le & \frac{s}{2}\int_{-1}^1\int_{-1}^1 \left\lvert \frac{p_3(\lvert s \pi xy \rvert)}{N} \right\rvert^2 dx dy 
	\leq   \frac{p_7(s)}{N^2},
	\end{split}
	\end{equation*}
	and similarly
	\begin{equation*}
	\begin{split}
	\big\lVert \mathcal{E}_{N-1} + \mathcal{S} \big\lVert_2^2 
	= & \int_{-s}^s\int_{-s}^s \left\lvert E_{N-1}(x,y) + S(x,y) \right\rvert^2 dx dy \\
	= & \frac{s}{2}\int_{-1}^1\int_{-1}^1 \left\lvert \pi^{\frac{1}{2}} \left(\frac{N}{2}\right)^{\frac{1}{4}} \text{ } \psi_{N-1}\left(\frac{s\pi xy}{\sqrt{2N}}\right) - \cos\left(s \pi xy - \frac{(N-1)\pi}{2}\right) \right\rvert^2 dx dy \\\le&\frac{s}{2}\int_{-1}^1\int_{-1}^1 \left\lvert \frac{p_3(\lvert s \pi xy \rvert)}{N} \right\rvert^2 dx dy 
	\leq  \frac{p_7(s)}{N^2},
	\end{split}
	\end{equation*}
    where $p_7$ is a polynomial of degree $7$ with non-negative coefficients.
	Taking the square roots of the norm then gives 
	\begin{equation*}
	\begin{split}
	\big\lVert \mathcal{E}_{N} + \mathcal{C} \big\lVert_2 \leq \frac{\sqrt{p_7(s)}}{N}
	\quad \text{ and } \quad
	\big\lVert \mathcal{E}_{N-1} + \mathcal{S} \big\lVert_2 \leq \frac{\sqrt{p_7(s)}}{N}.
	\end{split}
	\end{equation*}
    
	Bounding each of $\lVert \mathcal{E}_{N}\rVert_2$, $\lVert \mathcal{E}_{N-1}\rVert_2$, $\lVert \mathcal{C}\rVert_2$, and $\lVert \mathcal{S}\rVert_2$ is straightforward. We show how to bound $\lVert \mathcal{E}_{N}\rVert_2$ and $\lVert \mathcal{C}\rVert_2$, as the other two are almost the same. First,
	\begin{equation*}
	\begin{split}
	\lVert \mathcal{C}\rVert_2^2 
	& = \int_{-s}^s\int_{-s}^s \bigg[\frac{1}{\sqrt{2s}} \cos{\bigg(\frac{\pi xy}{s}\bigg)}\bigg]^2 dxdy \\
	& = \frac{s}{2} \int_{-1}^1\int_{-1}^1 \bigg[ \cos{(s \pi xy)}\bigg]^2 dxdy \leq s.
	\end{split}
	\end{equation*}
	Second, again using Proposition \ref{H2C}
	\begin{equation*}
	\begin{split}
	\lVert \mathcal{E}_{N}\rVert_2^2
	& = \frac{s}{2} \int_{-1}^{1}\int_{-1}^{1} \bigg[\pi^{\frac{1}{2}}\left(\frac{N}{2}\right)^{\frac{1}{4}}
	\psi_N\left(\frac{s\pi xy}{\sqrt{2N}}\right) \bigg]^2 dx dy \\
	& \leq \frac{s}{2} \int_{-1}^{1}\int_{-1}^{1} \bigg[ p_3(|\pi sxy|) + \left|\cos{\left(s \pi xy -\frac{N\pi}{2}\right)}\right| \bigg]^2 dx dy \leq p_7(s)
	\end{split}
	\end{equation*}
 for $p_7$ some degree $7$ polynomial with non-negative coefficients. 

 Combining all the results above, we see
 \begin{equation*}
     \|\mathcal{K}_{G_N}^{Bulk}-\mathcal{K}_{Sine}\|_1\le\frac{p_{10}(s)}{N}.
 \end{equation*}
 
 When $\frac{N}{2}$ is odd, similar result follows by applying Proposition \ref{H2C} to \eqref{GUE in1}.\qed

\subsection{LUE hard edge}\label{LUE he}
The outline is the same here as for the GUE bulk. We give how to decompose the operators, results needed to approximate the functions, and then put them together to prove the result.

\begin{proposition}\label{LUE fac}
	{\rm(LUE decomposition)} For a fixed $a\in\mathbb{N}$, let $I=[0,s]$ for some $s>0$, then we have factorization:
\begin{equation*}
\mathcal{K}^{Hard}_{L_{N,a}}=\mathcal{H}_N\mathcal{M}_N+\mathcal{M}_N\mathcal{H}_N
\end{equation*}
and
\begin{equation*}
\mathcal{K}_{Bes,a}=\mathcal{B}_a\mathcal{B}_a
\end{equation*}
where $\mathcal{H}_N$, $\mathcal{M}_N$ and $\mathcal{B}_a$ have kernels
\begin{equation*}
    H_N(x,y)=\sqrt{\frac{\tau_N\Gamma(N)}{8s\Gamma(N+a)}}\varphi_{a,N}\left(\frac{\tau_N xy}{4sN}\right),
\end{equation*}
\begin{equation*}
     M_N(x,y)=\sqrt{\frac{\tau_N\Gamma(N)}{8s\Gamma(N+a)}}\varphi_{a,N-1}\left(\frac{\tau_N xy}{4sN}\right)
\end{equation*}
and
\begin{equation*}
     B_a(x,y)=\frac{1}{2\sqrt s}J_a\left(\sqrt{\frac{xy}{s}}\right)
\end{equation*}
respectively on $I^2$. For any $t\ge 0$, $J_a(t)$ is the Bessel function with parameter $a$ and
\begin{equation*}
\varphi_{a,N}(t):=e^{-\frac{t}{2}}t^{\frac{a}{2}}L_N^a(t)
\end{equation*}
is the Laguerre function with degree $N$ and parameter $a$.
\end{proposition}
\begin{proof} Recall that the kernel of $\mathcal{K}_{L_{a,N}}^{Hard}$ is defined as $\frac{\tau_N}{4N}K_{L_{a,N}}\left(\frac{\tau_N x}{4N},\frac{\tau_N y}{4N}\right)$. From the integral form of the kernel $K_{L_{a,N}}$ on $[0,1]^2$ (\cite[Proposition $5.4.2$]{log-gas}):
\begin{equation*}
    K_{L_{a,N}}(x,y)=\frac{1}{2\Gamma(N)\Gamma(N+a)}\int_0^1\left(\phi_{a,N}(xu)\phi_{a,N-1}(yu)+\phi_{a,N-1}(xu)\phi_{a,N}(yu)\right)du
\end{equation*}
where $\phi_{a,N}(t):=\Gamma(N+1)e^{-\frac{t}{2}}t^{\frac{a}{2}}L_N^a(t)$, after rescaling, we see
\begin{equation*}
\begin{split}
     K&_{L_{a,N}}^{Hard}(x,y)=\frac{\tau_N}{4N}\int_0^1\frac{1}{2\Gamma(N)\Gamma(N+a)}\phi_{a,N}\left(\frac{\tau_N xu}{4N}\right)\phi_{a,N-1}\left(\frac{\tau_N yu}{4N}\right)du\\&~~+\frac{\tau_N}{4N}\int_0^1\frac{1}{2\Gamma(N)\Gamma(N+a)}\phi_{a,N}\left(\frac{\tau_N yu}{4N}\right)\phi_{a,N-1}\left(\frac{\tau_N xu}{4N}\right)du\\&=\frac{\tau_N\Gamma(N)}{8\Gamma(N+a)}\int_0^1\left(e^{-\frac{\tau_N xu}{8N}}\left(\frac{\tau_N xu}{4N}\right)^\frac{a}{2}L_N^a\left(\frac{\tau_N xu}{4n}\right)\right)\left(e^{-\frac{\tau_N yu}{8N}}\left(\frac{\tau_N yu}{4N}\right)^\frac{a}{2}L_{N-1}^a\left(\frac{\tau_N yu}{4N}\right)\right)du\\&~+\frac{\tau_N\Gamma(N)}{8\Gamma(N+a)}\int_0^1\left(e^{-\frac{\tau_N yu}{8N}}\left(\frac{\tau_N yu}{4N}\right)^\frac{a}{2}L_N^a\left(\frac{\tau_N yu}{4n}\right)\right)\left(e^{-\frac{\tau_N xu}{8N}}\left(\frac{\tau_N xu}{4N}\right)^\frac{a}{2}L_{N-1}^a\left(\frac{\tau_N xu}{4N}\right)\right)du.
\end{split}
\end{equation*}
Substitute $u$ by $\frac{u}{s}$ gives us the integral form on $[0,s]$ and hence, the decomposition of $\mathcal{K}_{L_{a,N}}^{Hard}$. The decomposition of $\mathcal{K}_{Bes,a}$ follows by the integral expression of $K_{Bes,a}$ (\cite[Exercise $7.2$]{log-gas}):
\begin{equation*}
K_{Bes,a}(x,y)=\frac{1}{4s}\int_0^s J_a\left(\sqrt{\frac{xu}{s}}\right)J_a\left(\sqrt{\frac{yu}{s}}\right)du.\qedhere
\end{equation*}\end{proof}

Before giving the exact proof, we first introduce three properties for the Bessel function.
\begin{proposition}\label{Bes}{\rm\cite[Chapter $2$, Equation $(9.14) \& (9.15)$]{olver1997asymptotics}}
    For a fixed complex number $a$, we have Bessel recurrences:
    \begin{equation*}
    \begin{array}{ll}
        2a J_a(t)=tJ_{a-1}(t)+t J_{a+1}(t),\\
        \\
        J'_a(t)=\frac{J_{a-1}(t)-J_{a+1}(t)}{2}.
    \end{array}
    \end{equation*}
\end{proposition}

\begin{proposition}\label{exL}
Given a fixed $a>0$, we have the expansion for the wighted Laguerre function:
   \begin{equation*}
\begin{split}
    \frac{e^{\frac{-x}{8n}}x^{\frac{a}{2}}L_n^a\left(\frac{x}{4n}\right)}{(2n)^a}&=J_a(\sqrt x)+\frac{a+1}{4n}\sqrt xJ_{a-1}(\sqrt x)-\frac{3a^2+5a+2}{96n^2}x J_a(\sqrt x)\\&~~-\frac{1}{96n^2}x^{\frac{3}{2}}J_{a-1}(\sqrt x)+\frac{3a^3+2a^2-3a-2}{48n^2}\sqrt x J_{a-1}(\sqrt x)+O(n^{-3})
    \end{split}
\end{equation*}
for large $n$. This holds uniformly in $x$ in a compact set on the positive half line.
\end{proposition}
\begin{proof}
    According to \cite[Lemma $1$]{forrester2019finite}, for any fixed $a>0$, the large $n$ expansion for the Laguerre polynomials is given as:
    \begin{equation}\label{ex4La}
       \begin{split}
           &\frac{x^{\frac{a}{2}}}{(2n)^a}L_n^a\left(\frac{x}{4n}\right)=J_a(\sqrt x)+\frac{1}{8n}\left(4a\sqrt xJ_{a-1}(\sqrt x)-xJ_{a-2}(\sqrt x)\right)\\
           &+\frac{1}{384n^2}\left(48a(a-1)xJ_{a-2}(\sqrt x)+(16-24a)\sqrt{x^3}J_{a-3}(\sqrt x)+3x^2J_{a-4}(\sqrt x)\right)+O\left(n^{-3}\right).
       \end{split} 
    \end{equation}
    This expansion holds uniformly in $x$ in a compact set on the positive half line. Multiply
    \begin{equation*}
        e^{-\frac{x}{8n}}=1-\frac{x}{8n}+\frac{x^2}{128n^2}+O(n^{-3})
    \end{equation*}
    on both sides of \eqref{ex4La}, the promised expansion follows by killing lower terms $J_{a-2}$, $J_{a-3}$ and $J_{a-4}$ according to the first recurrence in Proposition \ref{Bes}.
\end{proof}

\begin{proposition}\label{Integral}{\rm\cite[Equation $9.1.7$]{national_institute_of_standards_and_technology_nist_2010}}
    For a fixed $a\in\mathbb{N}$, when $t>0$ is small, we have
    \begin{equation*}
   J_a(t)\sim\frac{1}{\Gamma(a+1)}\left(\frac{t}{2}\right)^a,
    \end{equation*}
    and in particular, there exists a constant $C_a>0$ such that $|J_a(t)|\le C_a t^a$ for $t>0$ small.
\end{proposition}
\medbreak\noindent {\itshape Proof of Theorem \ref{main L hard}.}\enspace According to Proposition \ref{LUE fac} and Lemma \ref{21in},
\begin{equation}\label{LUE in1}
\begin{split}
\|\mathcal{K}_{L_{a,N}}^{Hard}&-\mathcal{K}_{Bes,a}\|_1=\|\mathcal{H}_N\mathcal{M}_N+\mathcal{M}_N\mathcal{H}_N-\mathcal{B}_a\mathcal{B}_a\|_1\\&\le\frac{1}{2}\|\mathcal{H}_N+\mathcal{M}_N-\sqrt2\mathcal{B}_a\|_2\|\mathcal{H}_N+\mathcal{M}_N+\sqrt2\mathcal{B}_a\|_2+\frac{1}{2}\|\mathcal{H}_N-\mathcal{M}_N\|^2_2.
\end{split}
\end{equation}
Recall $\tau_N=1-\frac{a}{2N}$. By the identity: $L_{N-1}^a(t)=L_N^a(t)-L_N^{a-1}(t)$ for Laguerre polynomials, we get
\begin{equation*}
    \begin{split}
       M_N(x,y)&=\sqrt{ \frac{\tau_N \Gamma(N)}{8s\Gamma(N+a)}}e^{-\frac{\tau_N xy}{8sN}}\left(\frac{\tau_N xy}{4sN}\right)^\frac{a}{2}L_{N-1}^{a}\left(\frac{\tau_N xy}{4sN}\right)\\&=\sqrt{ \frac{\tau_N N^a\Gamma(N)}{8s\Gamma(N+a)}}e^{-\frac{\tau_N xy}{8sN}}\left(\frac{\tau_N xy}{s}\right)^\frac{a}{2}L_{N}^a\left(\frac{\tau_N xy}{4sN}\right)\frac{1}{(2N)^a}\\&~~-\sqrt{ \frac{\tau_N N^a\Gamma(N)}{8s\Gamma(N+a)}}e^{-\frac{\tau_N xy}{8sN}}\left(\frac{\tau_N xy}{s}\right)^\frac{a}{2}L_{N}^{a-1}\left(\frac{\tau_N xy}{4sN}\right)\frac{1}{(2N)^a}\\&:=f_{\tau_N}(x,y)-g_{\tau_N}(x,y).
    \end{split}
\end{equation*}
Further define:
\begin{equation*}
\begin{array}{ll}
       S_{\tau_N}(x,y):=H_N(x,y)+M_N(x,y)=2f_{\tau_N}(x,y)-g_{\tau_N}(x,y),\\
       \\
    D_{\tau_N}(x,y):=H_N(x,y)-M_N(x,y)=g_{\tau_N}(x,y).
\end{array}
\end{equation*}
Let us now focus on the expansion for $S$ and $D$.

First, by the Taylor expansion, it is easy to see $\sqrt{\tau_N}=1-\frac{a}{4N}-\frac{a^2}{32N^2}+O(N^{-3})$. In addition, by routine calculations,
\begin{equation*}
\begin{split}
\sqrt{\frac{N^a\Gamma(N)}{\Gamma(a+N)}}&=\sqrt{\frac{N^a}{N(N+1)...(N+a-1)}}=\sqrt{\frac{1}{(1+\frac{1}{N})...(1+\frac{a-1}{N})}}\\&=\left(1+\frac{a(a-1)}{2N}+\frac{a(a-1)(a-2)(3a-1)}{24N^2}+O(N^{-3})\right)^{-\frac{1}{2}}\\&=1-\frac{a(a-1)}{4N}-\frac{a(a-1)(15a^2-47a+16)}{24(4N)^2}+O(N^{-3}).
\end{split}
\end{equation*}
Thus, the constant term can be written as:
\begin{equation}\label{con4B}
    \sqrt{\frac{\tau_N N^a\Gamma(N)}{\Gamma(a+N)}}=1-\frac{a^2}{4N}-\frac{15a^4-86a^3+99a^2-16a}{24(4N)^2}+O\left(N^{-3}\right).
\end{equation}

Notice that $f_{\tau_N}(x,y)=\sqrt{\tau_N}f_1(\sqrt{\tau_N} x,\sqrt{\tau_N} y)$ and $g_{\tau_N}(x,y)=\sqrt{\tau_N}g_1(\sqrt{\tau_N} x,\sqrt{\tau_N} y)$, we begin our discussion with $f_1$ and $g_1$. For convenience in notations, we first set $s:=1$.
Combining the expansion in Lemma \ref{exL} and \eqref{con4B}, we have:
\begin{equation}\label{a4f}
    \begin{split}
    \sqrt{8s\tau_N}& f_1(x,y)=J_a(\sqrt {xy})-\frac{a^2}{4N}J_a(\sqrt {xy})+\frac{a+1}{4N}\sqrt {xy}J_{a-1}(\sqrt {xy})\\&-\frac{3a^2+5a+2}{96N^2}(xy)J_a(\sqrt {xy})-\frac{a^2+3a+2}{48N^2}\sqrt {xy}J_{a-1}(\sqrt {xy})-\frac{(xy)^{\frac{3}{2}}}{96N^2}J_{a-1}(\sqrt {xy})\\&-\frac{15a^4-86a^3+99a^2-16a}{24(4N)^2}J_a(\sqrt{xy})+O(N^{-3}).
    \end{split}
\end{equation}
To further find $f_{\tau_N}$, we need an approximation for the Bessel function $J_a(\sqrt{xy\tau_N})$. In fact,
define $H_a(t):=J_a(\sqrt {x(1+t)})$, performing the Taylor expansion for small $t$ gives us:
\begin{equation}\label{T4B}
\begin{split}
    H_a(t)&=H_a(0)+H'_a(0)t+\frac{1}{2}H''_a(0)t^2+O(t^3)\\
    &=J_a(\sqrt x)+\frac{t}{2}\left(\sqrt x J_{a-1}(\sqrt x)-aJ_a(\sqrt x)\right)\\&~~+\frac{t^2}{8}\left((a^2+2a)J_a(\sqrt x)-xJ_a(\sqrt x)-2\sqrt x J_{a-1}(\sqrt x)\right)+O(t^3).
    \end{split}
\end{equation}
Replace $t$ by $-\frac{a}{2N}$ in \eqref{T4B} and plug all terms back into \eqref{a4f}. We finally get:
\begin{equation}\label{ex4f}
    \begin{split}
    \sqrt{8s}f_{\tau_N}(x,y)&=J_a(\sqrt {xy})+\frac{1}{4N}\sqrt{xy}J_{a-1}(\sqrt {xy})-\frac{27a^4-110a^3+99a^2-16a}{384N^2}J_a(\sqrt {xy})\\&-\frac{3a^3+7a^2+3a-1}{48N^2}\sqrt{xy}J_{a-1}(\sqrt {xy})+\frac{a^2-1}{16N^2}(xy)J_{a-1}(\sqrt {xy})\\&-\frac{6a^2+5a+2}{96N^2}(xy) J_{a}(\sqrt {xy})-\frac{1}{96N^2}(xy)^{\frac{3}{2}}J_{a-1}(\sqrt {xy})+O(N^{-3})
    \end{split}
\end{equation}
and similarly,
\begin{equation}\label{ex4g}
    \begin{split}
    \sqrt{8s}g_{\tau_N}(x,y)&=\frac{1}{2N}\sqrt {xy} J_{a-1}(\sqrt {xy})-\frac{3a^2-4a-1}{8N^2}\sqrt {xy} J_{a-1}(\sqrt {xy})\\&~~+\frac{2a-1}{8N^2}(xy)J_a(\sqrt{xy})+O(N^{-3}).
    \end{split}
\end{equation}
Recall $B_a(x,y)=\frac{1}{2\sqrt s}J_a\left(\sqrt{\frac{xy}{s}}\right)$.
Replace $x$ and $y$ by $\frac{x}{\sqrt s}$ and $\frac{y}{\sqrt s}$ respectively in \eqref{ex4f} and \eqref{ex4g}, we go back to our original expression and hence, get:
\begin{equation*}
    H_N(x,y)+M_N(x,y)-\sqrt 2B_a(x,y)=O(N^{-2}),~~H_\tau(x,y)-M_\tau(x,y)=O(N^{-1})
\end{equation*}
for any $(x,y)\in[0,s]^2$. It follows that the promised order has been proved.

To find the bound of $h(s)$ in Theorem \ref{L hard} when $s$ small, we only need to bound the $L^2$ norm for terms with the highest order in the coefficient of $\frac{1}{N^2}$. Actually, when $s$ is small, using Proposition \ref{Integral}, we see that the required terms of $2f_{\tau_N}-g_{\tau_N}-\sqrt2 J_a$ are $\frac{1}{\sqrt s}J_a\left(\sqrt{\frac{xy}{s}}\right)$ and $\frac{\sqrt{xy}}{s}J_{a-1}\left(\frac{\sqrt{xy}}{s}\right)$.
Notice $\frac{xy}{s}$ is small for small $s$, by Proposition \ref{Integral} again, we get
\begin{equation*}
\int_0^s\int_0^s \frac{1}{s}J^2_a\left(\sqrt{\frac{xy}{s}}\right)dxdy\le\frac{1}{s}\int_0^s\int_0^s\left(\frac{xy}{s}\right)^{a}dxdy\le s^{a+1}
\end{equation*}
 and
\begin{equation*}
\int_0^s\int_0^s \left(\frac{xy}{s^2}\right)J^2_{a-1}\left(\sqrt{\frac{xy}{s}}\right)dxdy\le\frac{1}{s}\int_0^s\int_0^s\left(\frac{xy}{s}\right)^{a}dxdy\le s^{a+1}.
\end{equation*}
Hence, for small $s$, when $N$ is large enough, we claim there exists a constant $C_a$ such that
\begin{equation*}
\|\mathcal{H}_N+\mathcal{M}_N-\sqrt2\mathcal{B}_a\|_2=\|2f_\tau-g_\tau-\sqrt2J_a\|_{L^2}\le  \frac{C_as^{\frac{a+1}{2}}}{N^2}.
\end{equation*}
The same argument for $g_{\tau_N}$ gives us
\begin{equation*}
\|\mathcal{H}_N-\mathcal{M}_N\|_2=\|g_{\tau_N}\|_{L^2}\le  \frac{C'_as^{\frac{a+1}{2}}}{N}
\end{equation*}
for some constant $C_a'$. Applying \eqref{ex4f} and \eqref{ex4g} to $2f_{\tau_N}-g_{\tau_N}+\sqrt2J_a$, we similarly see the required terms of $H_N+M_N+\sqrt{2}B_a$ are again $\frac{1}{\sqrt s}J_a\left(\sqrt{\frac{xy}{s}}\right)$ and $\frac{\sqrt{xy}}{s}J_{a-1}\left(\frac{\sqrt{xy}}{s}\right)$. It follows that $\|\mathcal{H}_N+\mathcal{M}_N+\sqrt{2}\mathcal{B}_a\|_2\le s^{\frac{a+1}{2}}$ for $s$ small and $N$ large. Putting all these estimates back into \eqref{LUE in1}, we obtain the promised result.
\qed

\subsection{Soft edge}
As already mentioned, our soft edge results for the GUE, LUE, and JUE involved already available trace class estimates by different authors. Here we list those results we require.
\begin{proposition}[GUE Soft Edge Estimate \cite{johnstone2012fast}]
	\label{prop:GUE_edge_Johnstone_Ma_trace_class_estimate}
 	Let $\mathcal{K}_{G_N}^{Soft}$ and $\mathcal{K}_{Ai}$ be as in Theorem \ref{main G soft}. Then there exists an absolute constant $C$ such that
	\begin{equation*}
		\lVert \mathcal{K}_{G_N}^{Soft} - \mathcal{K}_{Ai}\rVert_1 \leq \frac{C e^{-s}}{N^{\frac{2}{3}}}.
	\end{equation*}
\end{proposition}
\begin{proposition}[LUE Soft Edge Estimate \cite{el2006rate}]
	\label{prop:LUE_edge_ElKaroui_trace_class_estimate}
	Let $\mathcal{K}_{L_{a,N}}^{Soft}$ and $\mathcal{K}_{Ai}$ be as in Theorem \ref{main L soft}. Assume $\frac{N+a(N)}{N}\to \gamma\in[1,\infty)$. 
	Then, $\forall s\in\mathbb{R}$, there exists an integer $n(s,\gamma)$ such that $\forall N>n(s,\gamma),$ 
	\begin{equation*}
	\lVert \mathcal{K}_{L_{a,N}}^{Soft} - \mathcal{K}_{Ai}\rVert_1 \leq \frac{C(s) e^{-s}}{N^{\frac{2}{3}}},
	\end{equation*}
	for some continuous, non-increasing function $C(s)$.
\end{proposition}
\begin{proposition}[JUE Soft Edge Estimate \cite{johnstone2008multivariate}]
	\label{prop:JUE_edge_Johnstone_trace_class_estimate}
	Let $s>-\infty$ and fix $[s,\infty)$. Let $\mathcal{K}_{J_{a,b,N}}^{Soft}$ and $\mathcal{K}_{Ai}$ be as in Theorem \ref{main J soft}. Assume that $\frac{a(N)}{N}\to \alpha\in(0,\infty)$ and $\frac{b(N)}{N}\to \beta\in[0,\infty)$. Then, $\forall s\in\mathbb{R}$,
	\begin{equation*}
	\lVert \mathcal{K}_{J_{a,b,N}}^{Soft} - \mathcal{K}_{Ai}\rVert_1 \leq \frac{C(s) e^{-\frac{s}{2}}}{N^{\frac{2}{3}}},
	\end{equation*}
	for some continuous, non-increasing function $C(s)$.
\end{proposition}

\medbreak\noindent {\itshape Proof of Theorems \ref{G soft}, \ref{L soft}, and \ref{J soft}.}\enspace
Plug the appropriate trace class estimates above into the formula from Lemma \ref{M and M}.
\qed

\bigskip
\noindent{\bf Acknowledgements. }The results of this paper are part of both of the authors' PhD theses supervised by Mark Meckes. The authors are indebted to Professor Mark Meckes for providing important discussions and valuable feedback.

\newpage

\bibliographystyle{plain}
\bibliography{my_references}

\end{document}